\documentclass{amsart}

\usepackage{amsmath}    % need for subequations
\usepackage{graphicx}   % need for figures
\usepackage{epstopdf}
\usepackage{verbatim}   % useful for program listings
\usepackage{amssymb}
\usepackage{amsthm}
\usepackage{mathrsfs}
\usepackage{bbm}
\usepackage{exscale}
\usepackage{enumerate}

\newtheorem{theorem}{Theorem}%[section]
\newtheorem{lemma}[theorem]{Lemma}
\newtheorem{corollary}[theorem]{Corollary}
\newtheorem{proposition}[theorem]{Proposition}

\newtheorem{remark}[theorem]{Remark}
\theoremstyle{remark}

\theoremstyle{definition}
\newtheorem{example}[theorem]{Example}

\newcommand{\finsum}[3]{
\underset{#1=#2}{\overset{#3}\sum}}

\newcommand{\dint}{
\displaystyle\int}

\newcommand{\dsum}{
\displaystyle\sum}

\newcommand{\dpiifrac}{
\dfrac{1}{2\pi}}

\newcommand{\inflim}[1]{
\underset{#1\to\infty}\lim}

\newcommand{\zsum}[1]{
\underset{#1\in\Z}\dsum}
\newcommand{\zsumzero}[1]{
\underset{#1\neq0}\dsum}

\newcommand{\zsumn}[1]{
\underset{#1\in\Z^n}\dsum}
\newcommand{\zsumd}[1]{
\underset{#1\in\Z^d}\dsum}
\newcommand{\zsumnzero}[1]{
\underset{#1\neq 0}\dsum}

\newcommand{\sgn}{
\textnormal{sgn}}

\newcommand{\R}{
\mathbb{R}}

\newcommand{\N}{
\mathbb{N}}
\newcommand{\Z}{
\mathbb{Z}}

\newcommand{\supp}{
\textnormal{supp}}

\newcommand{\sinc}{
\textnormal{sinc}}
\newcommand{\schwartz}{
\mathscr{S}}
\newcommand{\bracket}[1]{
\left\langle#1\right\rangle}

\newcommand{\I}{
\mathscr{I}}

\newcommand{\phica}{
\widehat{\phi_{\alpha,c}}}

\newcommand{\scamfhat}{
\widehat{\I_{\alpha,c}^mf}}

\newcommand{\scamf}{
\I_{\alpha,c}^mf}

\newcommand{\scafhat}[1]{
\widehat{\I_{\alpha,c}^{#1}f}}

\newcommand{\scahat}{
\widehat{\I_{\alpha,c}f}}

\newcommand{\Lachat}{
\widehat{L_{\alpha,c}}}

 \newcommand{\eps}{
 \varepsilon}

 \newcommand{\floor}[1]{
 \lfloor#1\rfloor}
 \newcommand{\ceiling}[1]{
 \lceil#1\rceil}
 \title{Cardinal Interpolation With General Multiquadrics}

\author{Keaton Hamm}
\address{Department of Mathematics, Vanderbilt University,Nashville, TN, 37240, USA}
\email{keaton.hamm@vanderbilt.edu} 

\author{Jeff Ledford}
\address{Department of Mathematics, Virginia Commonwealth University, Richmond, VA, 23284, USA}
\email{jpledford@vcu.edu}

 \thanks{The first author was partially supported by National Science Foundation grant DMS 1160633.  The second author was partially supported by the 2014 Workshop in Analysis and Probability at Texas A\& M University.}

\begin{document}

%%%%%%%%%%%%%%%%%%%%%%%%%%%%%%%%%%%%%%%%%%%%%%%%%%%%%%%%%%%%%%%%%%%%%%%%%%%%
%          Abstract                                               %%%%%%%%%%
%%%%%%%%%%%%%%%%%%%%%%%%%%%%%%%%%%%%%%%%%%%%%%%%%%%%%%%%%%%%%%%%%%%%%%%%%%%%
\begin{abstract}

This paper studies the cardinal interpolation operators associated with the general multiquadrics, $\phi_{\alpha,c}(x) = (\|x\|^2+c^2)^\alpha$, $x\in\R^d$.  These operators take the form
$$\I_{\alpha,c}\mathbf{y}(x) = \zsumd{j}y_jL_{\alpha,c}(x-j),\quad\mathbf{y}=(y_j)_{j\in\Z^d},\quad x\in\R^d,$$
where $L_{\alpha,c}$ is a fundamental function formed by integer translates of $\phi_{\alpha,c}$ which satisfies the interpolatory condition $L_{\alpha,c}(k) = \delta_{0,k},\; k\in\Z^d$.

We consider recovery results for interpolation of bandlimited functions in higher dimensions by limiting the parameter $c\to\infty$.  In the univariate case, we consider the norm of the operator $\I_{\alpha,c}$ acting on $\ell_p$ spaces as well as prove decay rates for $L_{\alpha,c}$ using a detailed analysis of the derivatives of its Fourier transform, $\widehat{L_{\alpha,c}}$.

\end{abstract}
  \maketitle
\allowdisplaybreaks

%%%%%%%%%%%%%%%%%%%%%%%%%%%%%%%%%%%%%%%%%%%%%%%%%%%%%%%%%%%%%%%%%%%%%%%%%%%%
%          Introduction                                           %%%%%%%%%%
%%%%%%%%%%%%%%%%%%%%%%%%%%%%%%%%%%%%%%%%%%%%%%%%%%%%%%%%%%%%%%%%%%%%%%%%%%%%

\section{Introduction}\label{SECIntroduction}

The term {\em cardinal interpolation} refers to interpolation of data given at the integer lattice (or multi-integer lattice in higher dimensions).  It was I. J. Schoenberg's work on cardinal spline interpolation that brought about an intense study of the subject.  Many avenues of study have been explored, including forming interpolation operators from translates of certain radial basis functions (RBFs).  Works by Buhmann, Baxter, Riemenschneider, and Sivakumar \cite{Baxter,Buhmann,RiemSiva} (see also \cite{RS2} and references therein) have explored many cardinal interpolation operators of this type.  Some of the radial basis functions that have been considered are the Gaussian kernel, the thin plate spline, the Hardy multiquadric, and the inverse multiquadric.

Radial basis cardinal interpolation also enjoys a strong connection with classical sampling theory, as evidenced by much of the aforementioned literature.  This connection continues to be explored in recent developments by the second author \cite{Ledford,Ledford_ellp}, and by parts of this article.  As this is one of the most interesting aspects of cardinal interpolation, we provide some of the motivation.  Recall the one-dimensional Whittaker--Kotelnikov--Shannon Sampling Theorem, which states that a bandlimited function, $f$, (say with band size $\pi$) can be recovered uniformly by the series $\sum_{j\in\Z}f(j)\sinc(x-j)$, where the $\sinc$ function is suitably defined so that it takes value 1 at the origin, and 0 at all the other integers.  One observation about this series is that the $\sinc$ function decays slowly (as $|x|^{-1}$), and so to approximate the series by truncation for example, one would have to use quite a lot of data points of $f$ to get a reasonable degree of accuracy.  

However, there is a way of approximating the $\sinc$ series above: we seek to replace the $\sinc$ function with a so-called {\em fundamental function} (or {\em cardinal function}), $L$, that preserves the property that $L$ takes value 1 at the origin and 0 at all other integers.  We then form a function 
$$\I f(x) = \sum_{j\in\Z}f(j)L(x-j).$$
The trade-off here is that while $\I f$ is not pointwise equal to the function $f$, it does interpolate $f$ at the integer lattice, and moreover, the fundamental function $L$ may be constructed so that it decays much more rapidly than the $\sinc$ function.  Precisely, one may construct a fundamental function from a given radial basis function, $\phi$, which has the form
$$L_\phi(x) = \zsum{j}a_j\phi(x-j).$$

In the case of the Gaussian kernel, $g_\lambda(x) = e^{-\lambda|x|^2}$, the fundamental function decays exponentially, whereas the fundamental function for the Hardy multiquadric, $\sqrt{|x|^2+c^2}$, decays as $|x|^{-5}$.  So we give up the pointwise equality of the WKS Sampling Theorem in exchange for a series that converges more rapidly, while also ensuring that $\I f$ is close to $f$ in the $L_2$ norm.

%%%%%%%%%
%%%%%%%%% changes
%%%%%%%%%

This paper primarily considers the fundamental functions and cardinal interpolation operators associated with general multiquadrics, $\phi_{\alpha,c}(x) = (\|x\|^2+c^2)^\alpha$, which have thus far only been considered for certain exponents $\alpha$.  Interpolation with fundamental functions has too long a history to recount here; however, \cite{Buhmannbook} offers a good introduction using radial basis functions.  Using \eqref{EQLchatdef} below as a starting point is especially popular since it allows one to solve problems in the Fourier transform domain.  Many authors have used similar techniques for various radial basis functions.  In \cite{RS2,RiemSiva,RS3}, Riemenschneider and Sivakumar proved several results pertaining to the Gaussian.  Multiquadric cardinal interpolation has been studied in a similar way by Buhmann and Micchelli \cite{buhmannmicchelli}, Baxter \cite{Baxter}, Baxter and Sivakumar \cite{BS}, Riemenschneider and Sivakumar \cite{RiemSiva}, among others.  Compactly supported radial basis functions have been studied by Buhmann \cite{Buhmann_compact_support} and Wendland \cite{Wendland_ACOM}.

The rest of the paper is laid out as follows.  Section \ref{SECBasic} provides the necessary preliminaries and a discussion of applications and calculations of the fundamental functions; Section \ref{SECbaxteranalogue} shows recovery results for cardinal interpolation of bandlimited functions in any dimension via interpolants of the form discussed above.  Section \ref{SECpropfundamental} contains decay rates and information about the univariate fundamental functions associated with the general multiquadrics for a broad range of exponents. In Section \ref{SECEastJournal}, we consider the cardinal interpolation operators acting on data in traditional sequence spaces and calculate decay rates, bounds on the operator norms, and also explore some convergence properties in terms of the parameter $c$.  Section \ref{SECExamples} provides some interesting concrete examples based on the theoretical results from the previous section. Finally, Section \ref{SECproofs} provides the technical proofs of the statements in Section \ref{SECpropfundamental}.

%%%%%%%%%%%%%%%%%%%%%
%%%%%%%%%%%%%%%%%%%%%
%%%%%%%%%%%%%%%%%%%%%
\section{Basic Notions}\label{SECBasic}

If $\Omega\subset\R$ is an interval, then let $L_p(\Omega)$, $1\leq p\leq\infty$, be the usual Lebesgue space over $\Omega$ with its usual norm.  If no set is specified, we mean $L_p(\R)$.  Similarly, denote by $\ell_p(I)$ the usual sequence spaces indexed by the set $I$; if no index set is given, we refer to $\ell_p(\Z)$.  We will use $\N_0$ to denote the natural numbers including 0.  
Let $\schwartz$ be the space of Schwartz functions on $\R^d$, that is the collection of infinitely differentiable functions $\phi$ such that for all multi-indices $\alpha$ and $\beta$,
$$\underset{x\in\R^d}\sup\left|x^\alpha D^\beta\phi(x)\right|<\infty\;.$$

The Fourier transform of a Schwartz function $\phi$ is given by
\begin{equation}\label{ft}
 \widehat{\phi}(\xi):=\int_{\R^d} \phi(x)e^{-i\bracket{\xi, x}}dx,\quad \xi\in\R^d.
\end{equation}
Thus the inversion formula is
\begin{equation}\label{ftinversion}
\phi(x) = \dfrac{1}{(2\pi)^d}\dint_{\R^d}\widehat{\phi}(\xi)e^{i\bracket{x,\xi}}d\xi,\quad x\in\R^d.
\end{equation}
In the event that these formulas do not hold, then the Fourier transform should be interpreted in the sense of tempered distributions. Recall that if $f$ is a tempered distribution, then its Fourier transform is the tempered distribution defined by $\bracket{\widehat{f},\phi}=\bracket{f,\widehat{\phi}}$, $\phi\in\schwartz$.

Let $\alpha\in\R$ and $c>0$ be fixed; then define the $d$-dimensional \textit{general multiquadric} by
\begin{equation}\label{EQgmcdef}
\phi_{\alpha,c}(x):=\left(\|x\|^2+c^2\right)^\alpha,\quad x\in\R^d,
\end{equation}
where $\|\cdot\|$ denotes the Euclidean distance on $\R^d$.

If $\alpha\in\R\setminus\N_0$, the generalized Fourier transform of $\phi_{\alpha,c}$ is given by the following (see, for example, \cite[Theorem 8.15]{Wendland}):
\begin{equation}\label{EQgmcft}
 \phica(\xi)=\dfrac{2^{1+\alpha}}{\Gamma(-\alpha)}\left(\dfrac{c}{\|\xi\|}\right)^{\alpha+\frac{d}{2}}K_{\alpha+\frac{d}{2}}(c\|\xi\|),\quad \xi\in\R^d\setminus\{0\},
\end{equation}
where \cite[p.185]{Watson}
\begin{equation}\label{EQbesseldef}
K_\nu(r) = \dfrac{\Gamma(\frac{1}{2})r^\nu}{2^\nu\Gamma(\nu+\frac{1}{2})}\dint_1^\infty e^{-rx}(x^2-1)^{\nu-\frac{1}{2}}dx,\quad \nu\geq0,r>0.
\end{equation}
$K_\nu$ is called the modified Bessel function of the second kind.  We note that both $\phi_{\alpha,c}$ and its Fourier transform are radial. It is also clear from the definition that $K$ is symmetric in its order; that is, $K_{-\nu}=K_\nu$ for any $\nu\in\R$.  If $\alpha\in\N_0$, then the generalized Fourier transform of $\phi_{\alpha,c}$ involves a measure and so cannot be expressed as a function.

\subsection{Fundamental Functions}
Now suppose that $\alpha\in\R\setminus\N_0$ is fixed.  To define the \textit{fundamental function} associated with the general
multiquadric, we first define the following function
\begin{equation}\label{EQLchatdef}
 \widehat{L_{\alpha,c}}(\xi):=\dfrac{\phica(\xi)}{\zsumd{j}\phica(\xi+2\pi j)},\quad \xi\in\R^d.
\end{equation}
We will see that $\widehat{L_{\alpha,c}}\in L_1(\R^d)$, and so the function
\begin{equation}\label{EQHLcdef}
L_{\alpha,c}(x):=\dfrac{1}{(2\pi)^d}\int_{\R^d}\widehat{L_{\alpha,c}}(\xi)e^{i\bracket{x,\xi}}
d\xi,\quad x\in\R^d,
\end{equation}
is well-defined and continuous.  Furthermore, we will show that $L_{\alpha,c}$ is a \textit{fundamental function}, also called a \textit{cardinal function}, which means that
\begin{equation}\label{EQHfundamental}
L_{\alpha,c}(j) = \delta_{0,j},\quad j\in\Z^d,
\end{equation}
where $\delta_{i,j}$ is the Kronecker delta.

Additionally, $L_{\alpha,c}$ has the form
\begin{equation}\label{EQHLctranslatesofphi}
L_{\alpha,c}(x) = \zsumd{j}c_j\phi_{\alpha,c}(x-j),\quad x\in\R^d.
\end{equation}

Throughout the paper, we will use $A$ to denote an absolute constant due to the use of $c$ as the shape parameter of the multiquadric.  The value of the particular constant may change from line to line, and we use subscripts to denote dependence on certain parameters when needed.

%%%%%%%%%%
%%%%%%%%%% changes
%%%%%%%%%%
\subsection{Evaluation of Fundamental Functions and Applications}
Interpolation schemes involving fundamental functions as in \eqref{EQHLctranslatesofphi} have been studied for quite some time, and there are many aspects to the theory.  For example, such methods enjoy applications to geoscience \cite{FornFlyerBook} and sampling theory \cite{Ledford}.  Recently, investigations have considered interpolation via radial kernels on manifolds \cite{HNSW1,HNSW2}.  For a Galerkin type method for solving PDEs using meshless interpolation on the sphere, see \cite{NRW}.  

Given the widespread applications of radial basis function approximation, it is of import to the computational community to determine stable ways of evaluating the approximants.  Consequently, there is a substantial literature dealing with accuracy and stability of different computational methodologies for radial basis function approximation.  We do not claim to list all of these methods, but at least a sampling is in order.  We note that approximating \eqref{EQHLctranslatesofphi} is typically very difficult, especially if one dilates the lattice.  One way around this is the use of indirect computational methods to approximate the RBF interpolant \cite{FFL,FL} (for a discussion specifically related to multiquadrics, see \cite{FW}).  Another technique involves a change of basis method \cite{BLM}, while work by Fasshauer and McCourt \cite{FM} uses an eigenfunction decomposition to provide stable reconstruction using Gaussians.

Another quite promising method has recently been considered in which so-called {\em local Lagrange functions} are used to approximate rather than the global ones \cite{FHNWW,HNRW}.  Many of these results revolve around the situation of interpolation at finitely many data sites, which is of a somewhat different nature than we are considering here.  For the interested reader, we also mention that these methods of cardinal interpolation have, at their core, a deep connection to the classical spline theory instigated by Schoenberg (see \cite{Schoenberg} and references therein) and continued by many followers.  The underlying principle is that many of the results in spline interpolation theory have natural analogues via RBFs, and the problem at hand may determine which method is more useful.

\subsection{Examples}
Here, we provide some brief illustrations of the fundamental functions we have mentioned above.  Figure \ref{FIGSinc1D} shows the univariate fundamental function associated with the inverse multiquadric and the sinc function for comparison.

\begin{figure}[h!]
$$\hspace{-2em}\includegraphics[scale = 0.30]{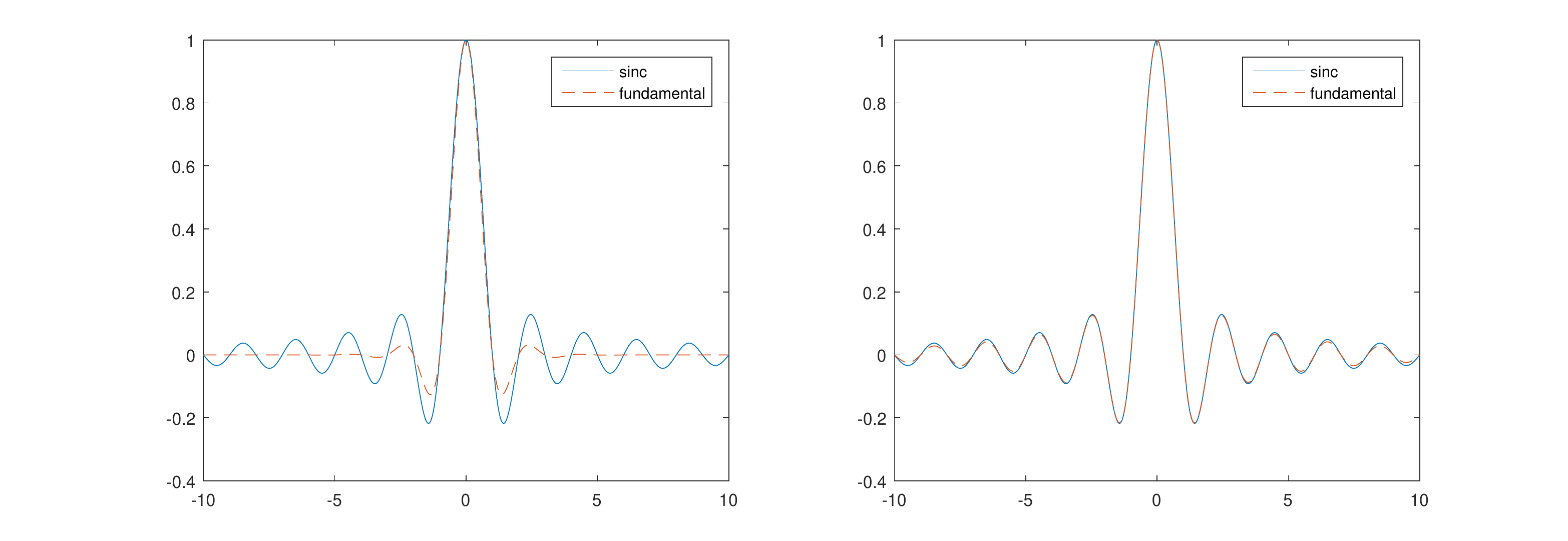}$$
\caption{Plots of sinc function and Fundamental function for the inverse multiquadric with $\alpha=-1/2$ with shape parameters $c=1$ (left) and $c=10$ (right).}\label{FIGSinc1D}
\end{figure}

\begin{figure}[h!]
$$\includegraphics[scale = 0.34]{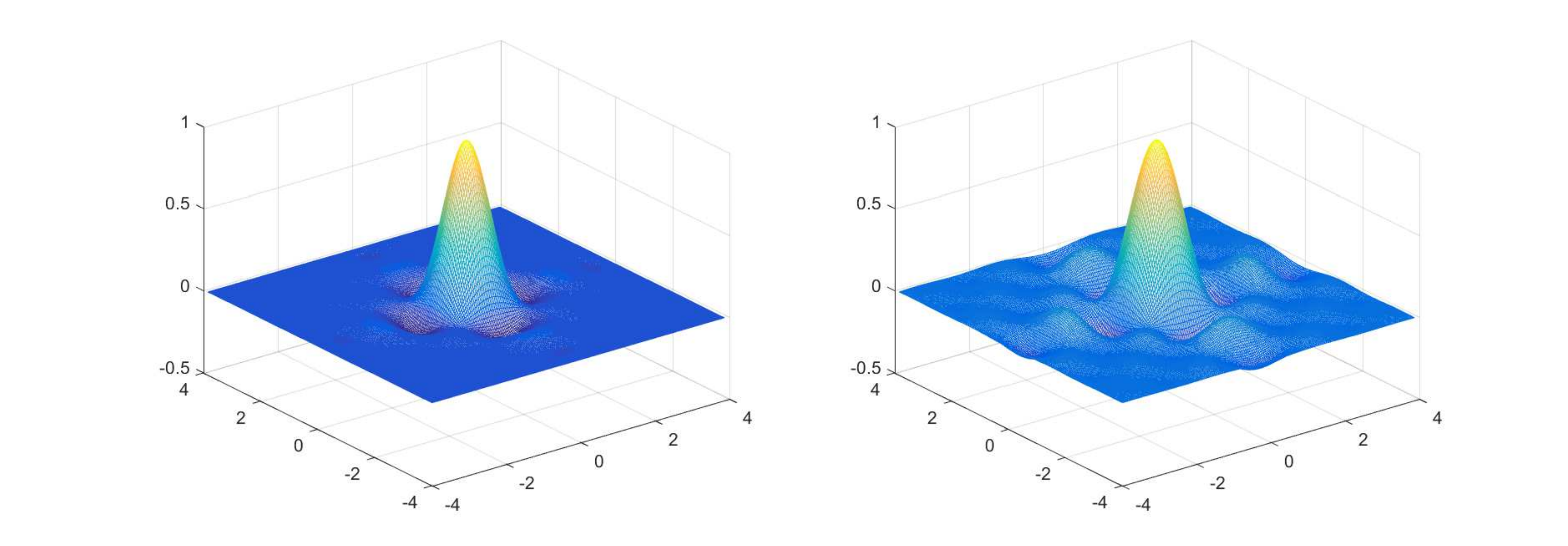}$$
\caption{Plots of 2-dimensional Poisson fundamental function $L_{-\frac{3}{2},1}$ (left) and sinc function (right) for comparison.}
\label{FIGsinc}
\end{figure}

The fundamental function was calculated by truncating the series in \eqref{EQLchatdef} and using a fast Fourier transform (FFT) method to approximate $L_{\alpha,c}$.  The same method may be applied for the multivariate version, as Figure \ref{FIGsinc} shows.
%%%%%%%%%%%%%%%%%%%%%%%%%%%%%%%%%%%%%%%%%%%%%%%%%%%%%%%%%%%%%%%%%%%%%%%%%%%%%%%
%%%              Baxter Analogue                              %%%%%%%%%%%%%%%%%
%%%%%%%%%%%%%%%%%%%%%%%%%%%%%%%%%%%%%%%%%%%%%%%%%%%%%%%%%%%%%%%%%%%%%%%%%%%%%%%
\section{Recovery of Multivariate Bandlimited Functions}\label{SECbaxteranalogue}
%%%%%%%%%%
%%%%%%%%%% changes
%%%%%%%%%%
When an interpolation scheme depends on a parameter, questions of convergence naturally arise.  This question has been addressed by several authors.  In \cite{Baxter}, Baxter examines the Hardy multiquadric, while the Gaussian is studied in \cite{RS3} by Riemenschneider and Sivakumar and in \cite{HMNW} by Hangelbroek, Narcowich, Madych, and Ward.  In a more general context,  `increasingly flat' radial basis functions are the focus of Driscoll and Fornberg in \cite{DF}, while the second author worked with `regular families' of cardinal interpolants in \cite{Ledford}.

In this section, we show that the result obtained by Baxter \cite{Baxter} holds not only for the traditional Hardy multiquadric (corresponding to $\alpha=1/2$) but rather for all $\alpha\in\R\setminus\N_0$.  We consider interpolation of bandlimited (or Paley-Wiener) functions in any dimension, and show that the cardinal interpolant converges to the function as the shape parameter $c$ tends to infinity. General multiquadrics were not considered for quite some time in this setting, but in \cite{Ledford} convergence results for cardinal interpolation of bandlimited functions were shown for a restricted range of exponents.  However, the analysis there was of a more general nature, so here we show that a more specific analysis yields convergence results for the full range $\alpha\in\R\setminus\N_0.$

We note that the results of this section have very close ties to classical sampling theory, which studies the reconstruction of certain classes of signals from their discrete samples.  As mentioned above, these considerations lead to alternative methods for approximating the sampling series given by the WKS Sampling Theorem for bandlimited signals.

Let $d$ be the dimension, and $\alpha\in\R\setminus\N_0$ and $c>0$ be fixed. It is evident from \eqref{EQgmcft} that $\phica$ does not change sign.  Therefore, $\Lachat(\xi)\geq0$ for all $\xi\in\R^d$.  From \eqref{EQLchatdef}, it is also evident that $0\leq\Lachat(\xi)\leq1$.  To show that $\Lachat\in L_1(\R^d)$ we begin with the following lemma.

\begin{lemma}\label{LEMbaxterphihatdecay}
 Let $R>r>0$, $c>0$, and $\alpha\in \R\setminus\N_0$. Then
 $$|\phica(R)|\leq e^{-c(R-r)}|\phica(r)|.$$
\end{lemma}

\begin{proof}
 Defining $\lambda:=\lambda_{c,\alpha,d}:=\frac{2^{1-\frac d2}\Gamma(\frac12)}{\Gamma(-\alpha)\Gamma(\alpha+\frac{d+1}{2})}c^{2\alpha+d}$, equations \eqref{EQgmcft} and \eqref{EQbesseldef} yield the following series of estimates:
 \begin{displaymath}
  \begin{array}{lll}
   |\phica(R)| & = & |\lambda|\dint_1^\infty e^{-cRt}(t^2-1)^{\alpha+\frac{d-1}{2}}dt\\
   \\
   & = & |\lambda|\dint_1^\infty e^{-c(R-r)t}e^{-crt}(t^2-1)^{\alpha+\frac{d-1}{2}}dt\\
   \\
   & \leq & |\lambda| e^{-c(R-r)}\dint_1^\infty e^{-crt}(t^2-1)^{\alpha+\frac{d-1}{2}}dt\\
   \\
   & = & e^{-c(R-r)}|\phica(r)|.
  \end{array}
 \end{displaymath}

\end{proof}

We note that if $\alpha+\frac{d}{2}\geq0$, then $(R/r)^{-\alpha-\frac{d}{2}}\leq1$, and so we have a purely exponential upper bound.

\begin{proposition}\label{PROPLchatintegrable}
 Let $\alpha\in\R\setminus\N_0$ and $c>0$. Then $\Lachat\in L_1(\R^d)$.
\end{proposition}

\begin{proof}
 First, choose an $M>0$ large.  Then since $|\Lachat(\xi)|\leq1$ for all $\xi$, we have that 
 $$\int_{[-M,M]^d}|\Lachat(\xi)|d\xi\leq (2M)^d.$$
 Now we need to estimate $$I:=\int_{\R^d\setminus[-M,M]^d}|\Lachat(\xi)|d\xi.$$  
 
 To do this, we establish a pointwise estimate for $\Lachat(\xi)$.  Let $\xi\in\R^d\setminus[-M,M]^d$ be fixed.  Since $M$ is large, there exists some $k_\xi\in\Z^d\setminus\{0\}$ such that $2\pi\leq\|\xi+2\pi k_\xi\|\leq 4\pi$.  Additionally, there is some constant $\gamma:=\gamma_{\alpha,d}>0$ which depends on $\alpha$ and $d$ such that if $cr\geq1$, $K_{\alpha+\frac{d}{2}}(cr)\geq\gamma e^{-cr}(cr)^{-\frac{1}{2}}$ (see, for example, \cite[Corollary 5.12]{Wendland}).  Therefore, choose $M$ large enough so that for $\xi\in\R^d\setminus[-M,M]^d$, we have $c\|\xi\|\geq1$.  Then if $\lambda$ is the constant from Lemma \ref{LEMbaxterphihatdecay},
\begin{align*}\left|\zsumd{k}\phica(\xi+2\pi k)\right|&\geq|\phica(\xi+2\pi k_\xi)| \\ & \geq \gamma|\lambda|\|\xi+2\pi k_\xi\|^{-\alpha-\frac{d}{2}}e^{-c\|\xi+2\pi k_\xi\|}(c\|\xi+2\pi k_\xi\|)^{-\frac{1}{2}}.
\end{align*}
 Now depending on the sign of $\alpha+\frac{d}{2}$, the above expression is minimized by plugging in $2\pi$ or $4\pi$ for $\|\xi+2\pi k_\xi\|$ in the appropriate places.  Consequently, there is a positive constant $D:=D_{c,\alpha,d}$ such that 
 $$\left|\zsumd{k}\phica(\xi+2\pi k)\right|\geq D e^{-4\pi c}.$$
 We also find from \cite[Lemma 5.13]{Wendland} that for every $r>0$, $K_{\nu}(r)\leq\sqrt{2\pi}\,r^{-\frac{1}{2}}e^{-r}e^\frac{|\nu|^2}{2r}.$  Consequently, by adjusting $M$ if need be so that $e^\frac{|\alpha+\frac{d}{2}|^2}{2c\|\xi\|}\leq2$ for $\xi\in\R^d\setminus[-M,M]^d$, we find that there is a positive constant $\beta$ such that  $K_{\alpha+\frac{d}{2}}(c\|\xi\|)\leq\beta e^{-c\|\xi\|}$. We conclude that
 \begin{align*}
 I &\leq D^{-1}e^{4\pi c}\dint_{\R^d\setminus[-M,M]^d}|\phica(\xi)|d\xi\\ &\leq\beta D^{-1}|\lambda|e^{4\pi c}\dint_{\R^d\setminus[-M,M]^d}\|\xi\|^{-\alpha-\frac{d}{2}}e^{-c\|\xi\|}d\xi.\end{align*}
 The integral on the right hand side above is convergent, and the constants outside are all finite, so $\Lachat\in L_1(\R^d)$.
\end{proof}

Now we turn our attention to the function $L_{\alpha,c}$.

\begin{proposition}\label{PROPLfundamental}
 Let $\alpha\in\R\setminus\N_0$ and $c>0$.  Then the function
 $$L_{\alpha,c}(x)=\dfrac{1}{(2\pi)^d}\int_{\R^d}\Lachat(\xi)e^{i\bracket{x,\xi}}d\xi$$
 is continuous, square-integrable, and satisfies the interpolatory condition $L_{\alpha,c}(k)=\delta_{0,k}$, for every $k\in\Z^d$.
\end{proposition}

\begin{proof}
 Proposition \ref{PROPLchatintegrable} implies that $L_{\alpha,c}$ is continuous and square-integrable, and indeed that $\Lachat$ is its Fourier transform.  To see the interpolatory condition, first define $Q_d:=[-\pi,\pi]^d$.  Then we have via the substitution $u=\xi+2\pi\ell$ that
 \begin{displaymath}
  \begin{array}{lll}
   L_{\alpha,c}(k) & = & \dfrac{1}{(2\pi)^d}\dint_{\R^d}\dfrac{\phica(\xi)}{\zsumd{j}\phica(\xi+2\pi j)}e^{i\bracket{k,\xi}}d\xi\\
   \\
   & = & \dfrac{1}{(2\pi)^d}\zsumd{\ell}\dint_{Q_d+2\pi\ell}\dfrac{\phica(\xi)}{\zsumn{j}\phica(\xi+2\pi j)}e^{i\bracket{k,\xi}}d\xi\\
   \\
   & =  & \dfrac{1}{(2\pi)^d}\dint_{Q_d}\zsumd{\ell}\dfrac{\phica(u-2\pi\ell)e^{-i\bracket{k,2\pi\ell}}}{\zsumd{j}\phica(u-2\pi\ell+2\pi j)}e^{i\bracket{k,u}}du\\
   \\
   & = & \dfrac{1}{(2\pi)^d}\dint_{Q_d} e^{i\bracket{k,u}}du\\
   \\
   & = & \delta_{0,k}.\\
  \end{array}
 \end{displaymath}
 The interchange of sum and integral in the third line is justified by the Dominated Convergence Theorem, for example.
\end{proof}
 
It is an important observation that much of the cardinal interpolation theory for bandlimited functions revolves around the fact that the fundamental functions converge to the function $\sin(\pi x)/(\pi x)$, which is equivalent to the statement that the Fourier transform of the fundamental function converges almost everywhere to the indicator function of the cube $[-\pi,\pi]^d$.  The story is no different here.  Defining $I(\xi)$ to be the function that takes value 1 whenever $\xi\in[-\pi,\pi]^d$, and 0 elsewhere, the following holds.
 
 \begin{proposition}\label{PROPconvergencetosinc}
 Let $\alpha\in\R\setminus\N_0$.  Then
 $$\inflim{c}\Lachat(\xi)=I(\xi)$$
 for all $\xi\in\R^d$ such that $\max\{|\xi_1|,\dots,|\xi_d|\}\neq\pi$.
\end{proposition}

\begin{proof}
 First suppose that $\xi\notin[-\pi,\pi]^d$.  Then there exists some $k_0\in\Z^d$ such that $\|\xi+2\pi k_0\|<\|\xi\|$.  Therefore by Lemma \ref{LEMbaxterphihatdecay}, 
 \begin{displaymath}
  \begin{array}{lll}
|\phica(\xi)| & \leq & e^{-c(\|\xi\|-\|\xi+2\pi k_0\|)}|\phica(\xi+2\pi k_0)|\\ \\
 & \leq & e^{-c(\|\xi\|-\|\xi+2\pi k_0\|)}\zsumd{k}|\phica(\xi+2\pi k)|.\\   
  \end{array}
 \end{displaymath}
 Consequently, since $\phica$ is of one sign,
 $$0\leq\Lachat(\xi)\leq e^{-c(\|\xi|-\|\xi+2\pi k_0\|)}.$$
 Since the exponent is negative, the limit of the right hand side as $c\to\infty$ is 0.  Therefore, for $\xi\notin[-\pi,\pi]^d$, $\inflim{c}\Lachat(\xi)=0$.
 
 Now suppose that $\xi\in(-\pi,\pi)^d$.  Then for all $k\in\Z^d\setminus\{0\}$, $\|\xi\|<\|\xi+2\pi k\|$.  Due to \eqref{EQLchatdef}, we may write
  $$\Lachat(\xi)=\left(1+\zsumnzero{k}\dfrac{\phica(\xi+2\pi k)}{\phica(\xi)}\right)^{-1},$$
  and therefore it suffices to show that 
$$\inflim{c}\;\zsumzero{k}\dfrac{\phica(\xi+2\pi k)}{\phica(\xi)}=0.$$
 By Lemma \ref{LEMbaxterphihatdecay}, 
$$
0\leq\zsumzero{k}\dfrac{\phica(\xi+2\pi k)}{\phica(\xi)} \; \leq \; \zsumzero{k}e^{-c(\|\xi+2\pi k\|-\|\xi\|)}.$$
The series on the right hand side is convergent and dominated by the convergent series where $c$ is replaced by 1, so $$\inflim{c}\zsumzero{k}\dfrac{\phica(\xi+2\pi k)}{\phica(\xi)}=0$$ as desired.  Convergence of the series stems from the fact that the series is majorized by $\zsumzero{k}e^{-2c\pi\|k\|}$, which converges.  
\end{proof}

We now consider interpolation of bandlimited functions at the lattice $\Z^d$ by translates of the function $L_{\alpha,c}(x)$.  Define the $d$-dimensional Paley-Wiener space by
$$PW_\pi^{(d)}:=\{f\in L_2(\R^d): \widehat{f}=0 \textnormal{ a.e. outside } [-\pi,\pi]^d\}.$$
We begin our analysis with an $L_2$ version of the Poisson Summation Formula:

\begin{lemma}[cf. \cite{Baxter} Lemma 3.2]\label{LEMPSF}
 If $f\in PW_\pi^{(d)}$, then
 \begin{equation}\label{EQPSF}
  \zsumd{j}\widehat{f}(\xi+2\pi j)=\zsumd{j}f(j)e^{-i\bracket{j,\xi}},
 \end{equation}
where the second series is convergent in $L_2(\R^d)$.
\end{lemma}

\begin{lemma}\label{LEMcauchysequence}
 Let $f\in PW_\pi^{(d)}$.  For $m\in\N$, define
 $$\scamfhat(\xi):=\left(\underset{\|k\|_1\leq m}{\dsum} f(k)e^{-i\bracket{k,\xi}}\right)\Lachat(\xi),\quad\xi\in\R^d,$$
 where $\|k\|_1 = \finsum{i}{1}{d}|k_i|$ for $k\in\Z^d$.  Then $(\scamf)_{m\in\N}$ forms a Cauchy sequence in $L_2(\R^d)$.
\end{lemma}

\begin{proof}
 Define $Q_m:\R^d\to\R$ via $$Q_m(\xi)=\underset{\|k\|_1\leq m}{\dsum} f(k)e^{-i\bracket{k,\xi}}.$$  Thus, $\scamfhat(\xi)=Q_m(\xi)\Lachat(\xi)$.  From Lemma \ref{LEMPSF}, it is clear that $(Q_m)_{m\in\N}$ is a Cauchy sequence in $L_2[-\pi,\pi]^d$.  So 
 \begin{displaymath}
  \begin{array}{lll}
   \|\scamfhat-\scafhat{\ell}\|_{L_2(\R^d)}^2 & \leq & \dint_{\R^d}|Q_m(\xi)-Q_\ell(\xi)|^2\left(\Lachat(\xi)\right)^2d\xi\\
   \\
   & = & \zsumd{k}\dint_{[-\pi,\pi]^d}|Q_m(\xi+2\pi k)-Q_\ell(\xi+2\pi k)|^2 \\
   & & \hfill\times\left(\Lachat(\xi+2\pi k)\right)^2 d\xi\\
   \\
   & = & \dint_{[-\pi,\pi]^d}|Q_m(\xi)-Q_\ell(\xi)|^2\zsumd{k}\left(\Lachat(\xi+2\pi k)\right)^2 d\xi\\
   \\
   & \leq & \dint_{[-\pi,\pi]^d}|Q_m(\xi)-Q_\ell(\xi)|^2d\xi\;. \\
  \end{array}
 \end{displaymath}
The interchange of sum and integral is valid by Tonelli's Theorem, and the last inequality follows from the fact that
\begin{equation}\label{EQ Hamm Lsquared}
\zsumd{k}\left(\Lachat(\xi+2\pi k)\right)^2=\dfrac{\zsumd{k}\phica^2(\xi+2\pi k)}{\left(\zsumd{l}\phica(\xi+2\pi l)\right)^2}\leq1.
\end{equation}
We also used the fact that for $k\in\Z^d$, $Q_m(\xi+2\pi k)=Q_m(\xi)$.  We conclude that
$(\scamfhat)_{m\in\N}$ is a Cauchy sequence in $L_2(\R^d)$ because
$\|\scamfhat-\scafhat{\ell}\|_{L_2(\R^d)}\leq\|Q_m-Q_\ell\|_{L_2[-\pi,\pi]^d}$, and the latter is Cauchy.
 \end{proof}

 Lemmas \ref{LEMPSF} and \ref{LEMcauchysequence} allow us to define
 \begin{equation}\label{EQinterpolantdef}
  \scahat(\xi):=\Lachat(\xi)\zsumd{k}f(k)e^{-i\bracket{k,\xi}},
 \end{equation}
where the series is convergent in $L_2(\R^d)$.  By a periodization argument similar to that in the proof of Lemma \ref{LEMcauchysequence}, one can show that $\scahat\in L_1(\R^d)$.  Thus applying the Fourier inversion formula we see that
$$\I_{\alpha,c}f(x) = \zsumd{k}f(k)L_{\alpha,c}(x-k),\quad x\in\R^d.$$

\begin{theorem}\label{THMbandlimitedconvergence}
Let $\alpha\in\R\setminus\N_0$.  If $f\in PW_\pi^{(d)}$, then 
 $$\inflim{c}\|\I_{\alpha,c}f-f\|_{L_2(\R^d)}=0,$$
 and $\inflim{c}|\I_{\alpha,c}f(x)-f(x)|=0$ uniformly on $\R^d$.
\end{theorem}

\begin{proof}
 We will first prove uniform convergence.  The proof is the same as in \cite{Baxter}.  Again let $I(\xi)$ be the characteristic function of the cube.  Then we see by the inversion formula and the oft-exploited periodization argument, that
 \begin{displaymath}
  \begin{array}{lll}
   \I_{\alpha,c}f(x)-f(x) & = & \dfrac{1}{(2\pi)^d}\dint_{\R^d}\zsumd{k}\widehat{f}(\xi+2\pi k)\left(\Lachat(\xi)-I(\xi)\right)e^{-i\bracket{x,\xi}}d\xi\\
   \\
   & = & \dfrac{1}{(2\pi)^d}\dint_{[-\pi,\pi]^d}\widehat{f}(\xi)\zsumd{k}\left(\Lachat(\xi+2\pi k)-I(\xi+2\pi k)\right)\\ & & \hfill\times e^{-i\bracket{x,\xi+2\pi k}}d\xi.\\
  \end{array}
 \end{displaymath}
 Therefore, we find that
 \begin{displaymath}
  \begin{array}{lll}
   |\I_{\alpha,c}f(x)-f(x)| & \leq & \dfrac{1}{(2\pi)^d}\dint_{[-\pi,\pi]^d}|\widehat{f}(\xi)|\zsumd{k}\left|\Lachat(\xi+2\pi k)-I(\xi+2\pi k)\right|d\xi\\
   \\
   & = & \dfrac{1}{(2\pi)^d}\dint_{[-\pi,\pi]^d}|\widehat{f}(\xi)|\left(1-\Lachat(\xi)+\zsumzero{k}\Lachat(\xi+2\pi k)\right)d\xi.\\
  \end{array}
 \end{displaymath}

 But then by definition,
$$
   \zsumnzero{k}\Lachat(\xi+2\pi k)  =  \dfrac{\zsumd{k}\phica(\xi+2\pi k)-\phica(\xi)}{\zsumd{l}\phica(\xi+2\pi l)}
    =  1-\Lachat(\xi).
  $$
Therefore,
$$|\I_{\alpha,c}f(x)-f(x)|  \leq  2\dfrac{1}{(2\pi)^d}\dint_{[-\pi,\pi]^d}|\widehat{f}(\xi)|(1-\Lachat(\xi))d\xi.$$
As the integrand is non-negative and bounded by $2|\widehat{f}(\xi)|\in L_1[-\pi,\pi]^d$, and $\inflim{c}(1-\Lachat(\xi))=0$, the Dominated Convergence Theorem implies that
$$\inflim{c}|\I_{\alpha,c}f(x)-f(x)| = 0,\quad x\in\R^d.$$
The upper bound is independent of $x$, hence the convergence is uniform.

We now turn to the proof of $L_2$ convergence. By Parseval's Identity, it suffices to show that $\|\scahat-\widehat{f}\|_{L_2(\R^d)}\to0$.  This breaks up into two estimates.  We first show this for the cube $[-\pi,\pi]^d$.  Recall that since $(e^{-i\bracket{k,\cdot}})_{k\in\Z^d}$ is an orthonormal basis for $L_2[-\pi,\pi]^d$, we may write $\widehat{f}(\xi)=\zsumd{k}f(k)e^{-i\bracket{k,\xi}}$. Moreover, 
$$\|\widehat{f}\|_{L_2[-\pi,\pi]^d} = \|f(k)\|_{\ell_2(\Z^d)}.$$
 Thus using \eqref{EQinterpolantdef}, 
\begin{align*}
\|\scahat-\widehat{f}\|_{L_2[-\pi,\pi]^d}^2 & =  \dint_{[-\pi,\pi]^d}\left|\zsumd{k}f(k)(\Lachat(\xi)-1)e^{-i\bracket{k,\xi}}\right|^2d\xi \\ & = \dint_{[-\pi,\pi]^d}|\Lachat(\xi)-1|^2\left|\zsumd{k}f(k)e^{-i\bracket{k,\xi}}\right|^2d\xi.
\end{align*}
  
The right hand side is bounded by $4\|f(k)\|^2_{\ell_2(\Z^d)}$, and so by the Dominated Convergence Theorem and Proposition \ref{PROPconvergencetosinc}, $\inflim{c}\|\scahat-\widehat{f}\|_{L_2[-\pi,\pi]^d}=0$.

Now for the rest of the space, for $l=(l_1,l_2,\dots,l_d)\in\Z^d\setminus\{0\},$ define $Q_l = [-\pi-2\pi l_1,\pi-2\pi l_1]\times\dots\times[-\pi-2\pi l_d,\pi-2\pi l_d]$.  Then we see that since $f$ is bandlimited,
$$\|\scahat-\widehat{f}\|^2_{L_2(\R^d\setminus[-\pi,\pi]^d)} = \|\scahat\|^2_{L_2(\R^d\setminus[-\pi,\pi]^d)} =  \zsumnzero{l}\|\scahat\|_{L_2(Q_l)}^2.$$
Consequently,
\begin{displaymath}
 \begin{array}{lll}
\dint_{\R^d\setminus[-\pi,\pi]^d}|\scahat(\xi)|^2d\xi &  =  & \zsumnzero{l}\dint_{Q_l}\left|\Lachat(\xi)\zsumd{k}f(k)e^{-i\bracket{k,\xi}}\right|^2d\xi\\
\\
& = & \dint_{[-\pi,\pi]^d}\zsumnzero{l}|\Lachat(\xi+2\pi l)|^2\left|\zsumd{k}f(k)e^{-i\bracket{k,\xi}}\right|^2d\xi,\\
\\
 \end{array}
\end{displaymath}
 by the Monotone Convergence Theorem.
 
Recall that $0\leq\Lachat(\xi)\leq1$, so $|\Lachat(\xi+2\pi\ell)|^2\leq\Lachat(\xi+2\pi\ell)$, and as calculated above, $\zsumnzero{\ell}\Lachat(\xi+2\pi\ell)=1-\Lachat(\xi)$.  Consequently, the integrand is bounded by
$$|1-\Lachat(\xi)|\left|\zsumd{k}f(k)e^{-i\bracket{k,\xi}}\right|^2\leq 2\|f(k)\|_{\ell_2(\Z^d)}^2.$$
Therefore, the Dominated Convergence Theorem and Proposition \ref{PROPconvergencetosinc} imply that $\inflim{c}\|\scahat\|_{L_2(\R^d\setminus[-\pi,\pi]^d)}=0$, and the proof is complete.
\end{proof}

To illustrate the convergence given by Theorem \ref{THMbandlimitedconvergence} above, the following figure shows the inverse multiquadric interpolant of the function whose Fourier transform is $\widehat{g}(\xi)=\xi^2$ in dimension 1.  

\begin{figure}[h!]\label{FIGinterp}
 $$\hspace{-1em}\includegraphics[scale=0.24]{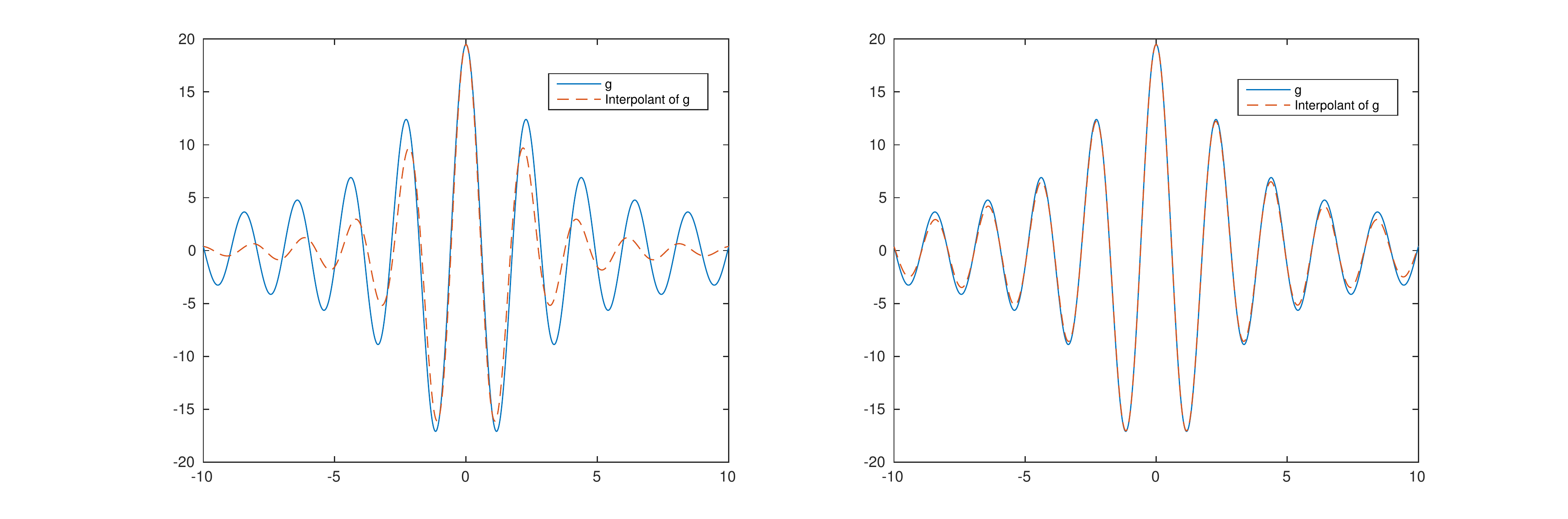}$$
 \caption{Plot of the function $g$ and its multiquadric interpolant for $\alpha = -1/2$ and both $c=1$ (left) and $c=10$ (right).}
\end{figure}
%%%%%%%%%%%%%%
%%%%%%%%%%%%%%
% 11/13 Update
%%%%%%%%%%%%%%
%%%%%%%%%%%%%%

\section{Properties of the Fundamental Function}\label{SECpropfundamental}

For the rest of the paper, we turn our attentions to the one-dimensional cardinal interpolation operator associated with the general multiquadric.  This section is devoted to the one-dimensional fundamental function $L_{\alpha,c}$, whose Fourier transform can be rewritten as
\begin{equation}\label{EQ l1}
\widehat{L_{\alpha,c}}(\xi)=\left[1+ \sum_{j\neq 0}  \dfrac{\widehat{\phi_{\alpha,c}}(\xi+2\pi j)    }{\widehat{\phi_{\alpha,c}}(\xi)} \right]^{-1}.
\end{equation}
The proofs of the results in this section are quite technical, so we postpone them until Section \ref{SECproofs} and simply state our conclusions here.

To determine decay rates for $L_{\alpha,c}$, we determine how many derivatives $\widehat{L_{\alpha,c}}$ has in $L_1$, which we accomplish by establishing pointwise estimates. We begin by fixing $\varepsilon \in [0,1) $, so that our estimates fall into three ranges: $|\xi|\leq\pi(1-\varepsilon)$, $\pi(1+\varepsilon)<|\xi|\leq 3\pi$, and the $2\pi$-length blocks $[(-2j-1)\pi,(-2j+1)\pi]$ for $|j|\geq 2$.  Due to the differing behavior of $\widehat{\phi_{\alpha,c}}$ for positive and negative values of $\alpha$, we must make corresponding distinctions in our calculations.

Following the insightful techniques of Riemenschneider and Sivakumar found in \cite{RiemSiva}, we begin by defining some auxiliary functions to aid in the analysis of the fundamental function.  We abbreviate \eqref{EQ l1} as $\widehat{L_{\alpha,c}}(\xi)=(1+s_{\alpha,c}(\xi))^{-1}$, where
\begin{equation}\label{EQ l2}
s_{\alpha,c}(\xi):=\sum_{j\neq 0} \widehat{\phi_{\alpha,c}}(\xi+2\pi j) / \widehat{\phi_{\alpha,c}}(\xi) =:\sum_{j\neq 0} a_j(\xi),
\end{equation}
and study the properties of $a_j$.

\begin{proposition}\label{PROP l1}
Suppose that $\alpha\in (0,\infty)\setminus \N$, $\varepsilon\in [0,1)$, $c\geq 1 $, and $k\in\mathbb{N}_0$. If $|\xi|\leq \pi(1-\varepsilon)$ and $k\leq 2\alpha+1$, then there exists a constant $A_{\alpha,k}(\varepsilon) >0$ such that
\begin{equation}\label{EQ l3}
|a_j^{(k)}(\xi)|\leq A_{\alpha,k}(\varepsilon)   c^{(k+1)(2\alpha-\lfloor\alpha\rfloor)+k}   e^{-2\pi c\varepsilon} e^{-2\pi c (|j|-1)},
\end{equation}
where $A_{\alpha,k}(\varepsilon)=O(1)$ as $\varepsilon\to 0$.
\end{proposition}

This estimate leads to the following bounds on $\widehat{L_{\alpha,c}}$ and its derivatives.

\begin{proposition}\label{PROP l2}
Suppose that $\alpha\in (0,\infty)\setminus \N$, $\varepsilon\in [0,1)$, $c\geq 1 $, and $k\in\mathbb{N}_0$. If $ k\leq 2\alpha+1$, then there exist constants $A_{\alpha,k}(\varepsilon), A_{\alpha,k} >0$ such that
\begin{enumerate}
\item[(i)] $|\widehat{L_{\alpha,c}}^{(k)}(\xi)|\leq A_{\alpha,k}(\varepsilon)   c^{2k(2\alpha-\floor\alpha\rfloor)+k}   e^{-2\pi c\varepsilon}$ whenever $|\xi|\leq \pi(1-\varepsilon)$,
\item[(ii)] $|\widehat{L_{\alpha,c}}^{(k)}(\xi)| \leq A_{\alpha,k}(\varepsilon) c^{(2k+1)(2\alpha-\lfloor\alpha\rfloor)+k} e^{-\pi c \varepsilon}$ whenever $|\xi|\in [(1+\varepsilon)\pi, 3\pi]$, and
\item[(iii)] $|\widehat{L_{\alpha,c}}^{(k)}(\xi)|\leq A_{\alpha,k}  c^{(2k+1)(2\alpha-\lfloor\alpha\rfloor)+k} e^{-2\pi c (|j|-1)} $ whenever $\xi\in[(-2j-1)\pi,\newline(-2j+1)\pi]$ and $|j|\geq 2$,
\end{enumerate}
where $A_{\alpha,k}(\varepsilon)=O(1)$ as $\varepsilon\to 0$.
\end{proposition}

These pointwise estimates yield the following result.

\begin{theorem}\label{THM l1}
Suppose that $\alpha\in (0,\infty)\setminus \N$, $c\geq 1 $, and $k\in\mathbb{N}_0$. If $ k\leq 2\alpha+1$, then there exists a constant $A_{\alpha,k} >0$ such that 
\begin{equation}\label{EQ l4}
\| \widehat{L_{\alpha,c}}^{(k)}  \|_{L_1(\mathbb{R})} \leq A_{\alpha,k} c^{(2k+1)(2\alpha-\floor{\alpha})+k}.
\end{equation}
\end{theorem}

Using standard arguments, we have the following estimate for the growth rate of $L_{\alpha,c}$.

\begin{corollary}\label{COR l3}
If $\alpha\in (0,\infty)\setminus \N$ and $c\geq 1$, then $L_{\alpha,c}(x)=O(|x|^{-\lfloor 2\alpha +1    \rfloor })$ as $|x|\to\infty$.
\end{corollary}

Analogous estimates can be made for the case that $\alpha<-1$.  Of interest to us are the following results.

\begin{theorem}\label{THM l2}
Suppose that $\alpha<-1$, $c\geq 1 $, and $k\in\mathbb{N}_0$. If $ k< 2|\alpha|-1$, then there exists a constant $A_{\alpha,k} >0$ such that 
\begin{equation}\label{EQ l5}
\| \widehat{L_{\alpha,c}}^{(k)}  \|_{L_1(\mathbb{R})} \leq A_{\alpha,k} c^{(2k+1)(2|\alpha|-\lfloor |\alpha| \rfloor-1)+k}.
\end{equation}
\end{theorem}

\begin{corollary}\label{COR l4}
If $\alpha<-1$ and $c\geq 1$, then $L_{\alpha,c}(x)=O(|x|^{-\lceil 2|\alpha| -2    \rceil })$ as $|x|\to\infty$.
\end{corollary}

It turns out that the Poisson kernel, which is the case $\alpha=-1$, is a special case, and exhibits much better decay because $\widehat{\phi_{-1,c}}$ is purely an exponential function.

\begin{theorem}\label{THM KPoissonFTdecay}
Suppose that $c\geq1$.  Then for every $k\in\N_0$, there exists a constant $A_k>0$ such that
\begin{equation}\label{EQ KPoissonFTdecay}
\|\widehat{L_{-1,c}}^{(k)}\|_{L_1(\R)}\leq A_kc^k.
\end{equation}
\end{theorem}

\begin{corollary}\label{COR KPoissondecay}
Suppose that $c\geq1$.  Then $L_{-1,c}(x)=O(|x|^{-k})$ as $|x|\to\infty$ for every $k\in\N_0$.
\end{corollary}

We may refine the above estimates in the case $k=1$ to find a uniform bound on the $L_1-$norm of $\widehat{L_{\alpha,c}}'$.
\begin{theorem}\label{THM l3}
Suppose that $\alpha\in \left((-\infty,-1]\cup (0,\infty)\right)\setminus\mathbb{N} $.  There exists a constant $A_{\alpha}>0$ such that for all $c\geq 1$, $\| \widehat{L_{\alpha,c}}'  \|_{L_1(\mathbb{R})}\leq A_\alpha$.
\end{theorem}

So far, the upper bounds on $L_{\alpha,c}$ may depend on the parameters $\alpha$ and $c$; however, the following still holds.

\begin{lemma}\label{LEM h3}
If $\alpha\in((-\infty,-1]\cup(0,\infty))\setminus\N$ and $c\geq1$, then
$|L_{\alpha,c}(x)|\leq1$ for all $x\in\R$.
\end{lemma}

We end the section with a statement on the zeros of $\widehat{L_{\alpha,c}}$.

\begin{theorem}\label{THM h3}
If $\alpha\in((-\infty,-1]\cup(0,\infty))\setminus\N$ and $c\geq1$, then $\widehat{L_{\alpha,c}}(2\pi k) = \delta_{0,k}$ for every $k\in\Z$.
\end{theorem}

%%%%%%%%%%%%%%%%%%%%%%%%%%%%%%%%%%%%%%%%%%%%%%%%%%%%%%%%%%%%%%%%%%%%%%%%%%%%%%%
%%%              East Journal Analogue                        %%%%%%%%%%%%%%%%%
%%%%%%%%%%%%%%%%%%%%%%%%%%%%%%%%%%%%%%%%%%%%%%%%%%%%%%%%%%%%%%%%%%%%%%%%%%%%%%%

\section{Norms and Convergence Properties of the One-dimensional Interpolation Operator}\label{SECEastJournal}

In this section, we show that the results of Riemenschneider and Sivakumar \cite{RiemSiva} have analogues for general multiquadrics.  In Section \ref{SECpropfundamental}, the decay of $\Lachat$ is discussed, and we use that information to uncover growth conditions on data that are suitable to cardinal interpolation. Recall from Corollary \ref{COR l3} and Theorem \ref{THM h3} that for $\alpha>0$ and $c\geq1$,
\begin{equation}\label{EQHLpositivealphabound}
|L_{\alpha,c}(x)| = O\left(\min\left\{1,|x|^{-\floor{2\alpha+1}}\right\}\right),\quad x\in\R.
\end{equation}
This decay is not the best one can get for individual $\alpha$, in fact for $\alpha=1/2$, Buhmann \cite{Buhmann} proves a decay rate of $|x|^{-5}$.  According to further work by Buhmann and Micchelli \cite{buhmannmicchelli}, it appears that the multiquadrics with exponents $(2k-1)/2$ for $k\in\N$ are exceptional cases.  For these values, $\Lachat$ can be shown to have more derivatives than what we have shown for general $\alpha$ due to some special symmetry involving the Bessel functions in the Fourier transforms. Moreover, in these cases, decay of the fundamental function is given by
\begin{equation}\label{EQHLoddbound}
\left|L_{\frac{2k-1}{2},c}(x)\right| = O\left(\min\left\{1,|x|^{-4k-1}\right\}\right),\quad x\in\R.
\end{equation}

For negative exponents, the so-called inverse multiquadrics, the fundamental functions have slightly slower decay (Corollary \ref{COR l4} and Theorem \ref{THM h3}):
\begin{equation}\label{EQHLnegativealphabound}
\left|L_{\alpha,c}(x)\right|=O\left(\min\left\{1,|x|^{-\ceiling{2|\alpha|-2}}\right\}\right),\quad x\in\R.
\end{equation}

As a consequence of the decay of the fundamental functions, we have the following.

\begin{proposition}\label{PROPRSprop33}
For $\alpha\in(-\infty,-3/2)\cup[1/2,\infty)\setminus\N$, the function
\begin{equation}\label{EQLambda}
\Lambda_{\alpha,c}(x):=\zsum{j}\left|L_{\alpha,c}(x+j)\right|,\quad x\in\R\end{equation}
is a well-defined, 1-periodic, bounded, continuous function.
\end{proposition}

\begin{proof}
Periodicity is apparent, so we need only consider $x\in[0,1]$.  If $\alpha\geq1/2$, then \eqref{EQHLpositivealphabound} gives
$$\Lambda_{\alpha,c}(x) = \sum_{j=-1}^1|L_{\alpha,c}(x+j)|+\sum_{|j|\geq2}|L_{\alpha,c}(x-j)| = O\left(1+\sum_{|j|\geq2}(|j|-1)^{-\floor{2\alpha+1}}\right),$$
which yields the result since $L_{\alpha,c}$ is continuous on $\R$ and $2\alpha+1\geq2$.

On the other hand, if $\alpha<-3/2$, then
$$\Lambda_{\alpha,c}(x) = O\left(1+\sum_{|j|\geq2}(|j|-1)^{-\ceiling{2|\alpha|-2}}\right).$$
As the exponent is less than -1, the series converges, whence the result.
\end{proof}

We next define the cardinal interpolation operator acting on sequences, and show that it is well-defined for sequences that grow at a sufficiently slower rate than the decay of the fundamental function. 

\begin{theorem}\label{THMRSthm34}
If $\alpha\in[1/2,\infty)\setminus\N$, and
$$|y_j|\leq A\left(1+|j|^{\floor{2\alpha+1}-1-\varepsilon}\right),\quad j\in\Z,$$
for some fixed positive constants $\varepsilon$ and $A$, then the function
\begin{equation}\label{EQdatainterpolant}
\I_{\alpha,c}\mathbf{y}(x):=\zsum{j}y_j L_{\alpha,c}(x-j),\quad\mathbf{y}=(y_j)_{j\in\Z}
\end{equation}
is well-defined, continuous on $\R$, and satisfies
$\I_{\alpha,c}\mathbf{y}(x) = O\left(1+|x|^{\floor{2\alpha+1}-1-\varepsilon}\right),$ $|x|\to\infty$.

If $\alpha\in(-\infty,-3/2)$, and $$|y_j|\leq A\left(1+|j|^{\ceiling{2|\alpha|-2}-1-\varepsilon}\right),\quad j\in\Z,$$
then $\I_{\alpha,c}$ as defined in \eqref{EQdatainterpolant} is well-defined, continuous on $\R$, and satisfies
$\I_{\alpha,c}\mathbf{y}(x) = O\left(1+|x|^{\ceiling{2|\alpha|-2}-1-\varepsilon}\right).$

\end{theorem}

\begin{proof}
Consider the case $\alpha\in[1/2,\infty)\setminus\N$.  We first show that $\I_{\alpha,c}\mathbf{y}$ is continuous on every interval of the form $[-M,M]$, for $M\in\N$.  Let $x\in[-M,M]$.  Then
$$\I_{\alpha,c}\mathbf{y}(x) = \sum_{|j|\leq2M}y_jL_{\alpha,c}(x-j)+\sum_{|j|>2M}y_jL_{\alpha,c}(x-j).$$
The first term on the right hand side is a finite sum of continuous functions, and so it suffices to show that the second sum converges uniformly on $[-M,M]$.  Indeed, the decays of $|y_j|$ and $L_{\alpha,c}$ yield the estimate
$$\sum_{|j|>2M}|y_jL_{\alpha,c}(x-j)|\leq A\sum_{|j|>2M}\dfrac{1+|j|^{\floor{2\alpha+1}-1-\varepsilon}}{|j|^{\floor{2\alpha+1}}},$$
which converges because $\varepsilon>0$ and ${2\alpha+1}\geq2$.

Now to consider the decay of $\I_{\alpha,c}$, let $|x|\geq1$ be fixed.  Then let $\nu:=\nu(x)$ be the unique integer such that $\nu(x)-1/2\leq x<\nu(x)+1/2$.  Since $\nu\neq0$, $|\nu|<2|x|$, and $|k-\nu|\leq2|x-k|$ for every $k\neq\nu$, we have
\begin{displaymath}
\begin{array}{lll}
|\I_{\alpha,c}(x)| & = & \left|y_\nu L_{\alpha,c}(x-\nu)+\dsum_{j\neq\nu}y_jL_{\alpha,c}(x-j)\right|\\
\\
& = & O\left(1+|\nu|^{\floor{2\alpha+1}-1-\varepsilon}+\dsum_{j\neq\nu}\dfrac{1+|j|^{\floor{2\alpha+1}-1-\varepsilon}}{|j-\nu|^{\floor{2\alpha+1}}}\right)\\
\\
& = & O\left(1+|\nu|^{\floor{2\alpha+1}-1-\varepsilon}\left(1+\zsumnzero{j}|j|^{-\floor{2\alpha+1}}\right)\right)\\
\\
& = & O\left(1+|x|^{\floor{2\alpha+1}-1-\varepsilon}\right).\\
\end{array}
\end{displaymath}

The proof for negative $\alpha$ values follows similar reasoning.
\end{proof}

Now we explore the properties of the cardinal interpolation operator $\I_{\alpha,c}$ acting on traditional sequence spaces; we first show boundedness, and follow with estimates on its norm.

\begin{theorem}\label{THMRSthm35}
For a fixed $\alpha\in(-\infty,-3/2)\cup[1/2,\infty)\setminus\N$ and $c>0$, the cardinal interpolation operator $\I_{\alpha,c}$ is a bounded linear operator from $\ell_p$ to $L_p$ for $1\leq p\leq\infty$.
\end{theorem}

\begin{proof}
The proof is the same as in \cite{RiemSiva}, which follows the techniques of \cite{MRR}.  Linearity is evident, and the cases $p=\infty$ and $p=1$ follow from Proposition \ref{PROPRSprop33} and Lemma \ref{LEM h3}, respectively.  Therefore, let $1<p<\infty$, and $x\in\R$ be fixed.  As before, let $\nu(x)$ be the unique integer such that $\nu(x)-1/2\leq x<\nu(x)+1/2$.  If $\mathbf{y}=(y_j)_{j\in\Z}\in\ell_p$, then Theorem \ref{THMRSthm34} implies that $\I_{\alpha,c}\mathbf{y}$ is continuous, and we write
$$\I_{\alpha,c}\mathbf{y}(x) = y_{\nu(x)}L_{\alpha,c}(x-\nu(x))+\dsum_{j\neq\nu}y_jL_{\alpha,c}(x-j).$$
To estimate the $L_p$ norm of the first term, notice that since $|L_{\alpha,c}(x)|\leq1$,
\begin{align*}\dint_\R\left|y_{\nu(x)}L_{\alpha,c}(x-\nu(x))\right|^pdx = \zsum{\ell}\dint_{\ell-\frac{1}{2}}^{\ell+\frac{1}{2}}\left|y_{\nu(x)}L_{\alpha,c}(x-\nu(x))\right|^pdx &\leq\zsum{\ell}|y_\ell|^p\\ &= \|y\|_{\ell_p}^p.\end{align*}

For the second term, first assume that $\alpha\in[1/2,\infty)\setminus\N$.  Then
\begin{displaymath}
\begin{array}{lll}
\dint_\R\left|\dsum_{j\neq\nu}y_jL_{\alpha,c}(x-j)\right|^pdx & = & \zsum{\nu}\dint_{\nu-\frac{1}{2}}^{\nu+\frac{1}{2}}\left|\dsum_{j\neq\nu}y_jL_{\alpha,c}(x-j)\right|^pdx\\
\\
& \leq & \zsum{\nu}\dint_{\nu-\frac{1}{2}}^{\nu+\frac{1}{2}}\left(\dsum_{j\neq\nu}|y_j||L_{\alpha,c}(x-j)|\right)^pdx\\
\\
& = & O\left(\zsum{\nu}\left(\dsum_{j\neq\nu}\dfrac{|y_j|}{|\nu-j|^{\floor{2\alpha+1}}}\right)^p\right).\\
\end{array}
\end{displaymath}
The quantity in the final line above corresponds to the $\ell_p$ norm of the discrete convolution $|\mathbf{y}|\ast\mathbf{b}$, where $|\mathbf{y}|=\left(|y_j|\right)_{j\in\Z}$, and the entries of $\mathbf{b}$ are given by $(1-\delta_{0,j})|j|^{-\floor{2\alpha+1}}$.  From \cite[p. 259, Theorem 7.6]{BennettSharpley}, we find that
\begin{equation}\label{EQdiscreteconvolution}
\||\mathbf{y}|\ast\mathbf{b}\|_{\ell_p}\leq\|\mathbf{y}\|_{\ell_p}\left[\underset{n\in\N}{\sup}\;n\;b^\sharp(n)\right],
\end{equation}
where $(b^\sharp(n))_{n\in\N}$ is a non-increasing rearrangement of $\mathbf{b}$.  It is easily checked that the supremum on the right hand side of \eqref{EQdiscreteconvolution} is finite, and therefore the second term estimated above is $O(\|y\|^p_{\ell_p})$, and the theorem follows.

The proof for the case $\alpha<-3/2$ is much the same, except that the power on the elements of $\mathbf{b}$ will be $-\ceiling{2|\alpha|-2}$, which nevertheless results in the supremum of $n\;b^\sharp(n)$ being finite.
\end{proof}

\begin{theorem}\label{THMRSthm36}
Let $\alpha\in(-\infty,-3/2)\cup[1/2,\infty)\setminus\N$ and $1<p<\infty$ be fixed.  Then 
$$\underset{c\geq1}\sup\;\|\I_{\alpha,c}\|_{\ell_p\to L_p}<\infty.$$
\end{theorem}
\begin{proof}
Consider a fixed $1<p<\infty$, and $p'$ such that $\frac{1}{p}+\frac{1}{p'}=1$.  Let $y\in\ell_p$, $g\in L_{p'}$, and $\nu(x)$ be the unique integer such that $\nu(x)-1/2\leq x<\nu(x)+1/2$.  Then
\begin{displaymath}
\begin{array}{lll}
\left|\dint_\R\I_{\alpha,c}y(x)g(x)dx\right| & = & \left|\dint_\R\zsum{j}y_jL_{\alpha,c}(x-j)g(x)dx\right|\\
\\
& \leq & \left|\dint_\R y_{\nu(x)}L_{\alpha,c}(x-\nu(x))g(x)dx\right| \\\\
& & \hfill+\left|\dint_\R\underset{j\neq\nu}\dsum y_jL_{\alpha,c}(x-j)g(x)dx\right|\\
\\
& =: & I_1+I_2.\\
\end{array}
\end{displaymath}
By H\"{o}lder's Inequality and the fact that $|L_{\alpha,c}(x)|\leq1$, $$I_1\leq\|y\|_{\ell_p}\|g\|_{L_{p'}}.$$
To estimate $I_2$, we represent $L_{\alpha,c}$ by its Fourier integral (see Proposition \ref{PROPLfundamental}) and integrate by parts.  Indeed,
\begin{displaymath}
\begin{array}{lll}
I_2 & = & \dpiifrac\left|\dint_\R\underset{j\neq\nu}\dsum y_j\dint_\R\Lachat(\xi)e^{i(x-j)\xi}d\xi g(x)dx\right|\\
\\
& = & \dpiifrac\left|\dint_\R\underset{j\neq\nu}\dsum\dint_\R\dfrac{y_j}{x-j}e^{i(x-j)\xi}\Lachat'(\xi)g(x)d\xi dx\right|\\
\\
& = & \dpiifrac\left|\dint_\R\Lachat'(\xi)\dint_\R\underset{j\neq\nu}\dsum\dfrac{y_j}{x-j}e^{-ij\xi}g(x)e^{ix\xi}dxd\xi\right|,\\
\end{array}
\end{displaymath}
where the final step follows from Fubini's Theorem.

From \cite[Proposition 1.3]{MRR}, the \textit{mixed Hilbert transform} defined by
\begin{equation}\label{EQmixedHilberttransform}
\mathscr{H}\mathbf{y}(x):=\underset{j\neq\nu(x)}\dsum\dfrac{y_j}{x-j}
\end{equation}
is a bounded linear operator from $\ell_p$ to $L_p$, and $\|\mathscr{H}\|_{\ell_p\to L_p}\leq A_p$, where $A_p$ is a constant depending only on $p$.  Consequently, 
\begin{displaymath}
\begin{array}{lll}
I_2 & = & \dpiifrac\left|\dint_\R\Lachat'(\xi)\dint_\R\mathscr{H}\left(ye^{-i(\cdot)\xi}\right)(x)g(x)e^{ix\xi}dxd\xi\right|\\
\\
& \leq & \dpiifrac\|\Lachat'\|_{L_1}\|\mathscr{H}\|_{\ell_p\to L_p}\|y\|_{\ell_p}\|g\|_{L_{p'}} \\
\\
& \leq & A_{\alpha,p}\|y\|_{\ell_p}\|g\|_{L_{p'}}.\\
\end{array}
\end{displaymath}
The final inequality comes from boundedness of the mixed Hilbert transform and Theorem \ref{THM l3}.  The conclusion of the theorem follows from the estimates on $I_1$ and $I_2$.
\end{proof}

We now estimate the norms in the cases $p=1$ and $p=\infty$. 

\begin{proposition}\label{PROPRS38}

Suppose $\alpha\in(-\infty,-3/2)\cup[1/2,\infty)\setminus\N$ and $c\geq1$ are fixed. The following hold:

(i) Let $\Lambda_{\alpha,c}$ be defined via \eqref{EQLambda}.  Then
$$\Lambda_{\alpha,c}(x)\leq A_\alpha\ln(c),\quad x\in\R,$$
where $A_\alpha>0$ is a constant.

(ii) $\|\I_{\alpha,c}\|_{\ell_\infty\to L_\infty}\leq A_\alpha\ln(c).$

(iii) $\|\I_{\alpha,c}\|_{\ell_1\to L_1}\leq A_\alpha\ln(c).$
\end{proposition}

\begin{proof}

We supply the proof for the case $\alpha\in[1/2,\infty)$, the estimates for the negative range of $\alpha$ being wholly similar. 

(i) By Theorems \ref{THM l1} and \ref{THM l3}, 
\begin{equation}\label{EQ Hamm Prop45bound}
|L_{\alpha,c}(x)|\leq A_\alpha\min\left\{1,\dfrac{1}{|x|},\dfrac{c^{5\alpha+2}}{|x|^2}\right\},\quad x\in\R,\quad c\geq1.
\end{equation}
Since $\Lambda_{\alpha,c}$ is 1-periodic, it suffices to check the inequality for $x\in[0,1]$.  Let $c\geq1$ and $N:=\ceiling{c^{5\alpha+2}}$.  Then by \eqref{EQ Hamm Prop45bound}
\begin{align*}
|\Lambda_{\alpha,c}(x)| &\leq  \zsum{j}|L_{\alpha,c}(x+j)|\\
&\leq  A_\alpha\left[1+\sum_{2\leq|j|\leq N} |x+j|^{-1}+\sum_{|j|>N}c^{5\alpha+2}|x+j|^{-2}\right]\\
&\leq  A_\alpha\left[1+\sum_{1\leq|j|\leq N-1}|j|^{-1}+\sum_{|j|\geq N}c^{5\alpha+2}|j|^{-2}\right]\\
&\leq  A_\alpha\left[1+\ln(N)+c^{5\alpha+2}N^{-1}\right],\\
\end{align*}
whence the result.

(ii) Simply note that if $\mathbf{y}\in\ell_\infty$, then $|\I_{\alpha,c}\mathbf{y}(x)|\leq\|\mathbf{y}\|_{\ell_\infty}\Lambda_{\alpha,c}(x)$, $x\in\R$, and apply (i).

(iii) Similarly, if $\mathbf{y}\in\ell_1$, then
$$\dint_\R\bigg|\zsum{j}y_jL_{\alpha,c}(x-j)\bigg|dx\leq\|\mathbf{y}\|_{\ell_1}\dint_0^1\Lambda_{\alpha,c}(x)dx,$$ whereby (i) provides the desired bound.

\end{proof}

We make note that the case $p=2$ provides an interesting special case.

\begin{theorem}\label{THMRSthm311}
For any $\alpha\in(-\infty,-3/2)\cup[1/2,\infty)\setminus\N$ and $c\geq1$,
$$\|\I_{\alpha,c}\|_{\ell_2\to L_2} = 1.$$
\end{theorem}
\begin{proof}
First, note that by Plancherel's Identity and a standard periodization argument,
\begin{multline}
\|\I_{\alpha,c}\mathbf{y}\|_{L_2}^2 = \frac{1}{2\pi}\|\widehat{\I_{\alpha,c}\mathbf{y}}\|_{L_2}^2 = \frac{1}{2\pi}\dint_\R\left|\zsum{j}y_je^{-ij\xi}\right|^2|\widehat{L_{\alpha,c}}(\xi)|^2d\xi \\= \frac{1}{2\pi}\dint_{-\pi}^\pi\left|\zsum{j}y_je^{-ij\xi}\right|^2\zsum{k}\left(\widehat{L_{\alpha,c}}(\xi+2\pi k)\right)^2d\xi.
\end{multline}
Consequently, $$\|\I_{\alpha,c}\mathbf{y}\|_{L_2}^2\leq\underset{\xi\in[-\pi,\pi]}\max\zsum{k}\left(\widehat{L_{\alpha,c}}(\xi+2\pi k)\right)^2.$$
Taking the supremum over $\mathbf{y}$ in the unit ball of $\ell_2$ yields
$$\|\I_{\alpha,c}\|_{\ell_2\to L_2} = \underset{\xi\in[-\pi,\pi]}\max\zsum{k}\left(\widehat{L_{\alpha,c}}(\xi+2\pi k)\right)^2.$$
That the maximum on the right hand side is at most 1 is the content of \eqref{EQ Hamm Lsquared}, but the maximum is attained at $\xi=0$ by Theorem \ref{THM h3}. 

\end{proof}

Having established some information about the interpolation operators and their norms for different values of $p$, we now shift our gaze to some convergence properties when the shape parameter $c\to\infty$ for a fixed $\alpha$.  As one might expect, for large (in absolute value) $\alpha$, we obtain better convergence results owing to the more rapid decay of the fundamental functions. 

To begin our discussion, consider the {\em Whittaker operator} defined via
\begin{equation}\label{EQwhittakerdef}
\mathcal{W}\mathbf{y}(x):=\zsum{j}y_j\dfrac{\sin(\pi(x-j))}{\pi(x-j)},\quad x\in\R,\quad \mathbf{y}=(y_j)_{j\in\Z}.
\end{equation}

This operator is bounded from $\ell_p$ to $L_p$ for $1<p<\infty$, and the following holds.

\begin{theorem}\label{THMconvergencetowhittaker}
If $\mathbf{y}\in\ell_p$, $1<p<\infty$, then for a fixed $\alpha\in(-\infty,-3/2)\cup[1/2,\infty)\setminus\N$,
$$\inflim{c}\|\I_{\alpha,c}\mathbf{y}-\mathcal{W}\mathbf{y}\|_{L_p}=0.$$
\end{theorem}
\begin{proof}
Boundedness of $\mathcal{W}$, Theorem \ref{THMRSthm36}, and the Uniform Boundedness Principle imply that it is sufficient to check convergence on the coordinate basis for $\ell_p$, $e_j:=\delta_{0,j}$.  That is, if suffices that
$$\|\I_{\alpha,c}e_j-\mathcal{W}e_j\|_{L_p} = \left\|L_{\alpha,c}(\cdot-j)-\dfrac{\sin\left(\pi(\cdot-j)\right)}{\pi(\cdot-j)}\right\|_{L_p} = \left\|L_{\alpha,c}-\dfrac{\sin(\pi\cdot)}{\pi\cdot}\right\|_{L_p}\to0,$$
 as $c\to\infty.$  By Theorem \ref{THMbandlimitedconvergence}, $|L_{\alpha,c}(x)-\sin(\pi x)/(\pi x)|$ converges to 0 uniformly as $c\to\infty$; moreover, Proposition \ref{THM l3} implies that $|\I_{\alpha,c}(x)|\leq A/|x|$, whence an application of the Dominated Convergence Theorem yields the statement of the theorem.
\end{proof}

Suppose a function $f$ has sufficiently slow growth so that $\I_{\alpha,c}\mathbf{y}$, with $\mathbf{y}:=\left(f(j)\right)_{j\in\Z}$, is well-defined (see Theorem \ref{THMRSthm34}).  In this case we let $\I_{\alpha,c}f(x):=\I_{\alpha,c}\mathbf{y}(x)$, and call this the cardinal interpolant of $f$ due to the identity
$$\I_{\alpha,c}f(j) = f(j),\quad j\in\Z.$$

We will consider pointwise and uniform convergence of $\I_{\alpha,c}f$ to $f$ as $c\to\infty$, but first we must make some preliminary arrangements.  Consider $\alpha$ to be fixed.  Then define
\begin{equation}\label{EQPhiRSdefinition}
 \Phi_{\alpha,c}(x,t):=\zsum{j}\Lachat(t+2\pi j)e^{-ix(t+2\pi j)},\quad t,x\in\R,
\end{equation}
and
\begin{equation}\label{EQPhiRSderivativedefinition}
 \Phi_{\alpha,c}^{(k)}(x,t):=\dfrac{\partial^k}{\partial t^k}\Phi_{\alpha,c}(x,t),\quad t,x\in\R.
\end{equation}

\begin{lemma}\label{LEMRSlem42}
 If $\alpha\in[1/2,\infty)\setminus\N$ and $k\in\{0,1,\dots,\floor{2\alpha+1}-2\}$, then
 $\Phi_{\alpha,c}^{(k)}$, is well-defined, continuous in each of its variables, and $2\pi$-periodic in the second variable.

 If $\alpha\in(-\infty,-3/2)$ and $k\in\{0,1,\dots, \ceiling{2|\alpha|-2}-2\}$, then $\Phi_{\alpha,c}^{(k)}$ is well-defined, continuous in each of its variables, and $2\pi$-periodic in the second variable.
 \end{lemma}
\begin{proof}
 Continuity of $\widehat{L_{\alpha,c}}^{(k)}$ is provided by Corollaries \ref{COR l3} and \ref{COR l4}, while Proposition \ref{PROP l2} shows that the series $\zsum{j}|\Lachat^{(k)}(t+2\pi j)|$ is uniformly convergent on $[-\pi,\pi]$.  Thus $\Phi_{\alpha,c}$ is well-defined and uniformly continuous for $t\in[-\pi,\pi]$, and moreover, we may differentiate the series term by term.  Finally, $2\pi$-periodicity is evident.
\end{proof}

Now let $C(\mathbb{T})$ be the space of continuous, $2\pi$-periodic functions on the real line, and let $M(\mathbb{T})$ denote its dual, which is the space of all $2\pi$-periodic complex Borel measures on the real line, with the total variation norm given by $\|\mu\|:=|\mu|([-\pi,\pi))$.  Following \cite[p.37]{ka}, given $\mu\in M(\mathbb{T})$, define the $j$-th Fourier-Stieltjes coefficient of $\mu$ by
\begin{equation}\label{EQfourierstieltjescoefficient}
 \widehat{\mu}(j):=\dint_{-\pi}^\pi e^{-ijt}d\mu(t),\quad j\in\Z,
\end{equation}
and the $n$-th Fej\'{e}r mean of the Fourier-Stieltjes series of $\mu$ by
\begin{equation}\label{EQfejermean}
 \sigma_n(\mu,t):=\finsum{j}{-n}{n}\left(1-\dfrac{|j|}{n+1}\right)\widehat{\mu}(j)e^{ijt},\quad t\in\R,\quad n\in\N_0.
\end{equation}

We begin by showing that for certain functions, the interpolant exhibits a special form.
\begin{theorem}\label{THMRSthm43}
 Let $\alpha$ be fixed, and suppose $f$ is given by
 \begin{equation}\label{EQthm43f}
 f(x):=(ix)^k\dint_{-\pi}^\pi e^{-ixt}d\mu(t),\quad x\in\R,  
 \end{equation}
 for some $\mu\in M(\mathbb{T})$ and some $k=0,1,\dots,\floor{2\alpha+1}-2$ in the case $\alpha\in[1/2,\infty)\setminus\N$ or $k=0,1,\cdots,\ceiling{2|\alpha|-2}-2$ in the case $\alpha\in(-\infty,-3/2)$. Let $\Phi_{\alpha,c}^{(k)}$ be defined by \eqref{EQPhiRSderivativedefinition}.  Then
 \begin{equation}\label{EQthm43I}
  \I_{\alpha,c}f(x)=\dint_{-\pi}^\pi\Phi_{\alpha,c}^{(k)}(x,t)d\mu(t),\quad c>0,\quad x\in\R.
 \end{equation}
 \end{theorem}
\begin{proof}
 By definition, $|f(x)|\leq|x|^k\|\mu\|$, thus Theorem \ref{THMRSthm34} implies that $\I_{\alpha,c}f$ is well-defined and continuous.  By property of the Fej\'{e}r means,
 \begin{equation}\label{EQTHM49 1}\I_{\alpha,c}f(x) = \inflim{n}\finsum{j}{-n}{n}f(j)L_{\alpha,c}(x-j) = \inflim{n}\finsum{j}{-n}{n}\left(1-\frac{|j|}{n+1}\right)f(j)L_{\alpha,c}(x-j).\end{equation}
 Therefore, by the inversion formula and a standard periodization argument,
 \begin{align}\label{EQTHM49 2}
\finsum{j}{-n}{n}f(j)L_{\alpha,c}(x-j) & = \dfrac{1}{2\pi}\dint_\R\left[\finsum{j}{-n}{n}\left(1-\frac{|j|}{n+1}\right)f(j)e^{ijt}\right]\Lachat(t)e^{-ixt}dt\nonumber\\
& = \dfrac{1}{2\pi}\dint_{-\pi}^{\pi}\left[\finsum{j}{-n}{n}\left(1-\frac{|j|}{n+1}\right)f(j)e^{ijt}\right]\Phi_{\alpha,c}(x,t)dt.
\end{align}
 
 By definition, $f(j) = (-ij)^k\widehat{\mu}(j)$ for every $j\in\Z$.  So if $k=0$, $f(j) = \widehat{\mu}(j)$, and if $k>0$, $f(0)=0$.  Consequently, if $k=0$, then
 
 $$\dpiifrac\dint_{-\pi}^\pi\left[\finsum{j}{-n}{n}\left(1-\frac{|j|}{n+1}\right)f(j)e^{ijt}\right]\Phi_{\alpha,c}(x,t)dt = \dpiifrac\dint_{-\pi}^\pi\sigma_n(\mu,t)\Phi_{\alpha,c}(x,t)dt.$$
 
 If $k>0$, then we integrate \eqref{EQTHM49 2} by parts $k$ times.  The boundary terms cancel due to periodicity, so
 \begin{align}\label{EQTHM49 3}
 & \dpiifrac\dint_{-\pi}^\pi\left[\finsum{j}{-n}{n}\left(1-\dfrac{|j|}{n+1}\right)f(j)e^{ijt}\right]\Phi_{\alpha,c}(x,t)dt \nonumber\\
 &= (-1)^k\dpiifrac\dint_{-\pi}^\pi\left[\dsum_{1\leq|j|\leq n}\left(1-\dfrac{|j|}{n+1}\right)\dfrac{f(j)}{(ij)^k}e^{ijt}\right]\Phi^{(k)}_{\alpha,c}(x,t)dt \nonumber\\
 &= \dpiifrac\dint_{-\pi}^\pi\left[\dsum_{1\leq|j|\leq n}\left(1-\dfrac{|j|}{n+1}\right)\widehat{\mu}(j)e^{ijt}\right]\Phi_{\alpha,c}^{(k)}(x,t)dt\nonumber\\
& = \dpiifrac\dint_{-\pi}^\pi\left[\finsum{j}{-n}{n}\left(1-\dfrac{|j|}{n+1}\right)\widehat{\mu}(j)e^{ijt}\right]\Phi_{\alpha,c}^{(k)}(x,t)dt,
 \end{align}
 where the final equality comes from the fact that $\int_{-\pi}^\pi\Phi_{\alpha,c}^{(k)}(x,t)dt = 0$ for $k\geq1$.
 
 Combining \eqref{EQTHM49 1}, \eqref{EQTHM49 2}, and \eqref{EQTHM49 3}, we see that
 $$\I_{\alpha,c}f(x) = \inflim{n}\dpiifrac\dint_{-\pi}^\pi\sigma_n(\mu,t)\Phi_{\alpha,c}^{(k)}(x,t)dt = \dint_{-\pi}^\pi\Phi_{\alpha,c}^{(k)}(x,t)d\mu(t).$$
\end{proof}

Using Theorem \ref{THMRSthm43}, we show the following result on uniform convergence.

\begin{theorem}\label{THMRSthm44}
 Suppose $\alpha$ is as above, and $f$ is given by
 $$ f(x):=(ix)^k\dint_{-\pi}^\pi e^{-ixt}d\mu(t),\quad x\in\R, $$
  for some $k$ as in Theorem \ref{THMRSthm43} and $\mu\in M(\mathbb{T})$ such that $$\supp(\mu)\cap[-\pi,\pi)\subset[-\pi(1-\varepsilon),\pi(1-\varepsilon)],$$
  for some fixed $0<\varepsilon<1$.  Then
  $$\inflim{c}\I_{\alpha,c}f(x) = f(x),\quad x\in\R,$$
  with convergence being uniform on compact subsets of $\R$.  In the case $k=0$, convergence is uniform on $\R$.
\end{theorem}

\begin{proof}
 Suppose first that $k\geq1$.  Let $M>0$ be fixed.  We will show that $\I_{\alpha,c}f(x)\to f(x)$ uniformly for $|x|\leq M$.  By Theorem \ref{THMRSthm43} and the definitions of $f$ and $\Phi_{\alpha,c}^{(k)}$,
 \begin{align*}
 \I_{\alpha,c}f(x)-f(x) & = \dint_{-\pi(1-\eps)}^{\pi(1-\eps)}\dfrac{\partial^k}{\partial t^k}\left(\Phi_{\alpha,c}(x,t)-e^{-ixt}\right)d\mu(t)\\
 & = \dint_{-\pi(1-\eps)}^{\pi(1-\eps)}\dfrac{\partial^k}{\partial t^k}\left(e^{-ixt}\left(\Lachat(t)-1\right)\right)d\mu(t)\\ & + \dint_{-\pi(1-\eps)}^{\pi(1-\eps)}\dfrac{\partial^k}{\partial t^k}\left(\dsum_{j\neq0}\Lachat(t+2\pi j)e^{-ix(t+2\pi j)}\right)d\mu(t)\\
 & =: \dint_{-\pi(1-\eps)}^{\pi(1-\eps)}I_{1,c}(t)d\mu(t) + \dint_{-\pi(1-\eps)}^{\pi(1-\eps)}I_{2,c}(t)d\mu(t).
 \end{align*}
 
 By applying the Leibniz rule,
 $$I_{2,c}(t) = (-ix)^ke^{-ixt}\left(\Lachat(t)-1\right)+\finsum{j}{0}{k-1}\binom{k}{j}(-ix)^{k-j}e^{-ixt}\Lachat^{(j)}(t).$$
 The first term above converges to 0 uniformly for $|t|\leq\pi(1-\eps)$, which can be seen from the last portion of the proof of Proposition \ref{PROPconvergencetosinc}.  Additionally, by Proposition \ref{PROP l2}(i) the second term above is bounded above by
 $$\finsum{j}{0}{k-j}A_{k,M,\eps,\alpha}c^{2j(2\alpha-\floor{\alpha})+j}e^{-2\pi c\eps},$$
 and hence converges uniformly to 0 as $c\to\infty$.  
 
 Similarly, write
 $$I_{2,c}(t) = \finsum{j}{0}{k}\binom{k}{j}(-ix)^{k-j}e^{-ixt}\left[\dsum_{l\neq0}e^{-i2\pi xl}\Lachat(t+2\pi l)\right].$$
 
 Here, if $|t|\leq\pi(1-\eps)$, then whenever $l=\pm1$ and $|l|\geq2$, $|t+2\pi l|$ falls into the respective ranges for Proposition \ref{PROP l2}(ii) and (iii).  Applying the estimates of that proposition demonstrates that $I_{2,c}(t)\to0$ uniformly as $c\to\infty$ for $|x|\leq M$.
 
 If $k=0$, then we again split the integral into two pieces, and analyze the corresponding integrands $I_{1,c}$ and $I_{2,c}$.  However, notice that there are no terms of the form $(ix)^{l}$, and so the inequalities from Proposition \ref{PROP l2} give upper bounds on the integrands that do not depend on $x$ at all, and so the convergence of $\I_{\alpha,c}$ to $f$ is uniform on $\R$.
\end{proof}

The condition on the support of the measure $\mu$ in the previous Theorem was essential because of the use of Proposition \ref{PROP l2} in the proof.  Heuristically, the condition on the support should be of no surprise due to the fact that we have no uniform control (in $c$) of the derivatives of $\Lachat$  at the boundary points $\pm\pi$.  In fact, it is likely that the derivatives get much larger near the boundary as $c$ grows, especially given the fact that $\Lachat$ converges to the characteristic function of $(-\pi,\pi)$.  Nevertheless, we may make a weaker assumption on $\mu$ which still yields a convergence result.

\begin{theorem}\label{THMRSthm46}
 Suppose $\alpha$ is as above, and $f$ is given by
 $$f(x):=\dint_{-\pi}^\pi e^{-ixt}d\mu(t),\quad x\in\R,$$
 for some $\mu\in M(\mathbb{T})$ which is absolutely continuous with respect to the Lebesgue measure.  Then
 $$\inflim{c}\I_{\alpha,c}f(x) = f(x),\quad x\in\R,$$
 with convergence being uniform on $\R$.
\end{theorem}
\begin{proof}
 The proof is quite similar to that of Theorem \ref{THMRSthm44}.  Let $\gamma>0$ be arbitrary, and choose $\eps>0$ such that $|\mu|[-\pi,-\pi(1-\eps)]+|\mu|[\pi(1-\eps),\pi]<\gamma$ since $|\mu|$ is absolutely continuous with respect to the Lebesgue measure.  Then, as before,
\begin{multline}
 \I_{\alpha,c}f(x)-f(x) = \dint_{-\pi}^\pi e^{-ixt}\left(\Lachat(t)-1\right)d\mu(t)\\+\dint_{-\pi}^\pi\left(\dsum_{j\neq0}\Lachat(t+2\pi j)e^{-ix(t+2\pi j)}\right)d\mu(t).
\end{multline}
 
 Split each integral into one over $[-\pi(1-\eps),\pi(1-\eps)]$ and one near the endpoints.  The integral over the interior segment can be handled exactly as in the proof of Theorem \ref{THMRSthm44}.  For the integrals near the endpoints, notice that $|e^{-ixt}(\Lachat(t)-1)|\leq2$, and so the first integral is at most $2(|\mu|[-\pi,-\pi(1-\eps)]+|\mu|[\pi(1-\eps),\pi])$, which is at most $2\gamma$ by the choice of $\eps$.  Meanwhile, by Proposition \ref{PROP l2}(iii), the second integrand is at most
 $$2+\dsum_{|j|\geq2}\left|\Lachat(t+2\pi j)\right|\leq2+A_\alpha\dsum_{|j|\geq2}c^{2\alpha-\floor{\alpha}}e^{-2\pi c(|j|-1)},$$
 which can be bounded by a constant depending only on $\alpha$.  Thus $|\I_{\alpha,c}f(x)-f(x)|\leq A_\alpha\gamma$ for some constant $A_\alpha$ independent of $c$, and so the conclusion follows.
\end{proof}

\begin{remark}
We conclude this section with the special consideration of the case $\alpha=-1$, where $\phi_{\alpha,c}$ is called the Poisson kernel.  As mentioned above, the fundamental function for the Poisson kernel, $L_{-1,c}$, decays faster than any polynomial.  Thus any results in this section that depend on the decay rate of the fundamental function hold true for the Poisson kernel.  In particular, Propositions \ref{PROPRSprop33} and \ref{PROPRS38}, and Theorems \ref{THMRSthm35}, \ref{THMRSthm36}, \ref{THMRSthm311}, \ref{THMconvergencetowhittaker}, and \ref{THMRSthm46} are all valid for the Poisson kernel, since existence of the interpolant primarily depends on the decay of the fundamental function.

Moreover, because the fundamental function for the Poisson kernel decays so rapidly, we find much stronger versions of the remaining results in this section.
\end{remark}

We begin by stating the analogue of Theorem \ref{THMRSthm34}.

\begin{theorem}\label{THMRSthm34Poisson}
Suppose that $$|y_j|\leq A(1+|j|^k),\quad j\in\Z$$ for some $k\in\N_0$ and constant $A$.  Then the {\em cardinal Poisson interpolant},
$$\I_{-1,c}\mathbf{y}(x):=\zsum{j}y_jL_{-1,c}(x-j),\quad\mathbf{y}=(y_j)_{j\in\Z},$$ is well-defined, continuous on $\R$, and satisfies $\I_{-1,c}\mathbf{y}(x) = O(1+|x|^l),\;|x|\to\infty$ for any $l\geq k$.

\end{theorem}

As a consequence of the preceding theorem, the function $\Phi_{-1,c}$ defined by \eqref{EQPhiRSdefinition} has well-defined derivatives of all orders.  In particular, Lemma \ref{LEMRSlem42} holds, and consequently we find the following analogue of Theorem \ref{THMRSthm43}.

\begin{theorem}
Suppose $f$ is given by
$$f(x):=(ix)^k\dint_{-\pi}^\pi e^{-ixt}d\mu(t),$$
for some $\mu\in M(\mathbb{T})$ and some $k\in\N_0$.  Then
$$\I_{-1,c}f(x) = \dint_{-\pi}^\pi \Phi_{-1,c}^{(k)}(x,t)d\mu(t),\quad c>0,\quad x\in\R.$$
\end{theorem}

Finally, Theorem \ref{THMRSthm44} holds for any $k\in\N_0$ for the Poisson kernel.

%%%%%%%%%%%%%%%%%%%%%%%%%%%%%%%%%%%%%%%%%%%%%%%%%%
%%%%%%    Examples                   %%%%%%%%%%%%%
%%%%%%%%%%%%%%%%%%%%%%%%%%%%%%%%%%%%%%%%%%%%%%%%%%
\section{Convergence Examples}\label{SECExamples}

In this section, we illustrate the convergence phenomena discussed in the previous section.  The examples are of a similar flavor to those found in \cite{RiemSiva}.

\begin{example}\label{EXAMPLE1}
Let $\alpha\in[1/2,\infty)\setminus\N$ and $k\in\{0,1,\dots,\floor{2\alpha+1}-2\}$, or $\alpha\in(-\infty,-3/2)$ and $k\in\{0,1,\dots,\ceiling{2|\alpha|-2}-2\}$, or $\alpha=-1$ and $k\in\N_0$.  Let $\mu_k$ be the $2\pi$-periodic extension of the measure $i^k\delta_0$, where $\delta_0$ is the usual Dirac measure at 0.  If
$$f_k(x):=(-ix)^k\dint_{-\pi}^\pi e^{-ixt}d\mu_k(t) = x^k,\quad x\in\R,$$
then Theorem \ref{THMRSthm44} implies that $\inflim{c}\I_{\alpha,c}f_k(x)=f_k(x)$ uniformly on compact subsets of $\R$.
\end{example}

However, Theorem \ref{THM h3} allows us to say more given this information.  If $k=0$, then we find the following identity on account of Theorem \ref{THMRSthm43}:

\begin{equation}
\I_{\alpha,c}f_0(x) = \Phi_{\alpha,c}(x,0) = \zsum{j}\Lachat(2\pi j)e^{-ix2\pi j} = 1.
\end{equation}
For higher order polynomials, we may use Theorems \ref{THMRSthm43} and \ref{THMRSthm44} to show that
$$\I_{\alpha,c}f_k(x)-x^k = i^k\finsum{l}{0}{k-1}\zsum{j}\Lachat^{(k-l)}(2\pi j)(-ix)^le^{-ix2\pi j},$$
whereby one can obtain an error bound in terms of $c$ via Proposition \ref{PROP l2}.  This also demonstrates that $\I_{\alpha,c}f_k\to f_k$ uniformly on compact subsets of $\R$.

\begin{example}
Let $0<a\leq\pi$ be fixed, and let $\mu$ be the $2\pi$-periodic extension of $\frac{1}{2a}\chi_{[-a,a]}dt$, where $\chi_{[-a,a]}$ takes value 1 on $[-a,a]$ and 0 elsewhere.  Define
$$f(x) = \dint_{-\pi}^\pi e^{-ixt}d\mu(t) = \frac{1}{2a}\dfrac{\sin(ax)}{ax},\quad x\in\R.$$
Since $\mu$ is absolutely continuous with respect to the Lebesgue measure, Theorem \ref{THMRSthm46} implies that $\I_{\alpha,c}f\to f$ uniformly on $\R$.  Note that this fact also follows from Theorem \ref{THMbandlimitedconvergence}, albeit from substantially different reasoning.  However, the hypothesis of Theorem \ref{THMbandlimitedconvergence} requires the function being interpolated to be bandlimited, and so the following result cannot be obtained simply by appealing to that theorem.
\end{example}

\begin{example}
Let $\alpha$ and $k$ be as in Example \ref{EXAMPLE1}.  Let $\mu_k$ be the periodic extension of $i^k\frac{1}{2a}\chi_{[-a,a]}dt$ for some $a<\pi$.  Then if
$$g_k(x):=(-ix)^k\dint_{-\pi}^\pi e^{-ixt}d\mu_k(t) = x^k\dfrac{\sin(ax)}{ax},\quad x\in\R,$$
Theorem \ref{THMRSthm44} implies that $\inflim{c}\I_{\alpha,c}g_k(x)= g_k(x)$, uniformly on compact subsets of $\R$.
\end{example}

\section{Proofs for Section \ref{SECpropfundamental}}\label{SECproofs}

In this section we prove the various results listed in Section \ref{SECpropfundamental}.  Our methods closely resemble those found in \cite{RiemSiva}.  To reduce the clutter in our calculations, we will henceforth drop explicit dependence upon $\alpha$ in our calculations that follow.  We begin by rewriting \eqref{EQgmcft} in terms of a Laplace transform to exhibit the singularity at the origin.  To do this, we require an integral representation for the Bessel function, which we find from \cite[p.185]{Watson}:
\begin{equation}\label{WatsonIntegral}
K_\nu(r) = \dfrac{\Gamma(\frac{1}{2})r^\nu}{2^\nu\Gamma(\nu+\frac{1}{2})}\dint_1^\infty e^{-rx}(x^2-1)^{\nu-\frac{1}{2}}dx,\quad \nu\geq0,r>0.
\end{equation}

Consequently, putting $r=c|\xi|$ and performing the substitution $x|\xi| = t+|\xi|$ yields the following on account of \eqref{EQgmcft}.

\begin{equation}\label{EQ JL1}
\widehat{\phi_{c}}(\xi)= A_\alpha c^{2\alpha+1}|\xi|^{-2\alpha-1}e^{-c|\xi|}\int_{0}^{\infty}e^{-ct}t^\alpha(t+2|\xi|)^\alpha  dt,
\end{equation}
where $A_\alpha$ is a constant.  We relabel the product of the exponential and the integral in the above expression $F_\alpha$.  That is,
\begin{equation}\label{EQ JL2}
F_\alpha(\xi):=e^{-c|\xi|}\int_{0}^{\infty}e^{-ct}t^\alpha(t+2|\xi|)^\alpha  dt,
\end{equation} 
hence \eqref{EQ JL1} may be abbreviated as
\begin{equation}\label{EQ JL3}
\widehat{\phi_{c}}(\xi)= A_\alpha c^{2\alpha+1}|\xi|^{-2\alpha-1}F_\alpha(\xi).
\end{equation}
For $\alpha <-1$, we have
\begin{equation}\label{EQ JL4}
\widehat{\phi_{c}}(\xi)= A_\alpha F_{|\alpha|-1}(\xi).
\end{equation}

We turn our attention to the estimates involving $F_\alpha$ and its derivatives, and begin by noting that
\begin{equation}
\label{EQ JL5}
F_\alpha(\xi) = e^{-c|\xi|}\mathcal{L}[t^\alpha(t+2|\xi|)^\alpha](c),
\end{equation}
where $\mathcal{L}$ denotes the usual Laplace transform, $\mathcal{L}[f](s)=\int_{0}^{\infty}f(t)e^{-st}dt$.  Our estimates for $F_\alpha$ will rely on estimates for the Laplace transform.  The first result in this direction is a lower bound.

\begin{lemma}\label{LEM JL1}
For $\alpha >0$,
\begin{equation}\label{EQ JL6}
\mathcal{L}[t^{\alpha}(t+2|\xi|)^\alpha](c)\geq \max\left\{ \dfrac{(2|\xi|)^\alpha \Gamma(\alpha+1)}{c^{\alpha+1}} , \dfrac{\Gamma(2\alpha+1)}{c^{2\alpha+1}}              \right\}.
\end{equation}
\end{lemma}
\begin{proof}
We obtain this bound by combining the inequalities
\[
\mathcal{L}[t^\alpha(t+2|\xi|)^\alpha](c) \geq \int_{0}^{\infty} e^{-ct} (2|\xi|)^\alpha t^{\alpha} dt = \dfrac{(2|\xi|)^\alpha \Gamma(\alpha+1)}{c^{\alpha+1}}
\]
and
\[
\mathcal{L}[t^\alpha(t+2|\xi|)^\alpha](c) \geq \int_{0}^{\infty} e^{-ct}  t^{2\alpha} dt = \dfrac{\Gamma(2\alpha+1)}{c^{2\alpha+1}}.
\]
\end{proof}

We state the upper bounds in the following lemma.
\begin{lemma}\label{LEM JL2}
For $\alpha\in (0,\infty)\setminus\mathbb{N}$ and $l\in\mathbb{N}_0$, the following estimates hold for the Laplace transform $\mathcal{L}[t^\alpha(t+2|\xi|)^{\alpha-l}](c)$.
\begin{enumerate}
\item[(i)] If $0\leq l < \lfloor \alpha \rfloor$, $$\quad\left|\mathcal{L}[t^\alpha(t+2|\xi|)^{\alpha-l}](c)\right|\leq \dfrac{(4|\xi|)^{\alpha-l}\Gamma(\alpha+1)}{c^{\alpha+1}} +\dfrac{2^{\alpha-l}\Gamma(2\alpha+1-l)}{c^{2\alpha+1-l}};$$

\item [(ii)] if $ \lfloor \alpha \rfloor < l < 2\alpha +1  $,$$\left|\mathcal{L}[t^\alpha(t+2|\xi|)^{\alpha-l}](c)\right|\leq \dfrac{2^{3\alpha-2l}|\xi|^{2\alpha-l}}{c} +\dfrac{2^{\alpha-l}\Gamma(2\alpha+1-l)}{c^{2\alpha+1-l}};$$

\item[(iii)] if $l=2\alpha+1$ (thus $\alpha$ is an odd half integer),
$$\left|\mathcal{L}[t^\alpha(t+2|\xi|)^{\alpha-l}](c)\right|\leq \ln\left( 1+|\xi|^{-1} \right) +\dfrac{e^{-c}}{c(1+|\xi|)}.$$
\end{enumerate}
\end{lemma}

\begin{proof}
For (i), we have the following
\begin{align*}
\mathcal{L}[t^\alpha(t+2|\xi|)^{\alpha-l}](c)  = &\int_{0}^{2|\xi|}e^{-ct} t^\alpha(t+2|\xi|)^{\alpha-l}  dt   +  \int_{2|\xi|}^{\infty}e^{-ct} t^\alpha (t+2|\xi|)^{\alpha-l}   dt \\
\leq &   \int_{0}^{2|\xi|}e^{-ct} t^\alpha(4|\xi|)^{\alpha-l}  dt   +  \int_{2|\xi|}^{\infty}e^{-ct} t^\alpha(2t)^{\alpha-l}   dt \\
\leq &\dfrac{ (4|\xi|)^{\alpha-l}\Gamma(\alpha+1)}{c^{\alpha+1}}  + \dfrac{2^{\alpha-l}\Gamma(2\alpha+1-l)}{c^{2\alpha+1-l}}.
\end{align*}
Inequality (ii) follows from a similar calculation, replacing $t$ with $2|\xi|$ in the first integral below:
\begin{align*}
\mathcal{L}[t^\alpha(t+2|\xi|)^{\alpha-l}](c)  = &\int_{0}^{2|\xi|} e^{-ct}t^\alpha(t+2|\xi|)^{\alpha-l}  dt   +  \int_{2|\xi|}^{\infty}e^{-ct} t^\alpha (t+2|\xi|)^{\alpha-l}   dt \\
\leq &   \int_{0}^{2|\xi|}e^{-ct}2^{3\alpha -2l}|\xi|^{2\alpha -l}  dt   +  \int_{2|\xi|}^{\infty} e^{-ct}2^{\alpha-l}t^{2\alpha-l}   dt \\
\leq &\dfrac{2^{3\alpha -2l}|\xi|^{2\alpha -l}  }{c}  + \dfrac{2^{\alpha-l} \Gamma(2\alpha+1-l)}  {c^{2\alpha+1-l}}.
\end{align*}
Finally, for inequality (iii), we have
\begin{align*}
\mathcal{L}[t^\alpha(t+2|\xi|)^{-\alpha-1}](c) =& \int_{0}^{\infty}e^{-ct}\dfrac{t^\alpha}{(t+2|\xi|)^{\alpha+1}}dt = \int_{0}^{\infty}e^{-c|\xi|t}\dfrac{t^\alpha}{(t+2)^{\alpha+1}}  dt\\
=&\int_{0}^{|\xi|^{-1}}e^{-c|\xi|t}\dfrac{t^\alpha}{(t+2)^{\alpha+1}}dt    +\int_{|\xi|^{-1}}^{\infty}e^{-c|\xi|t}\dfrac{t^\alpha}{(t+2)^{\alpha+1}}dt \\
\leq &  \int_{0}^{|\xi|^{-1}}  (1+t)^{-1}dt   +   \int_{|\xi|^{-1}}^{\infty} e^{-c|\xi|t} (1+t)^{-1}dt\\
\leq & \ln\left(1+|\xi|^{-1}\right) + \dfrac{e^{-c}}{c(1+|\xi|)} \leq \ln\left(1+|\xi|^{-1}\right) +  \dfrac{e^{-c}}{c}.
\end{align*}
\end{proof}

Thus we have the following Lemma which allows us to easily bound the derivatives of $F_\alpha$.

\begin{lemma}\label{LEM JL3}
For $\alpha>0$, and $0\leq k\leq 2\alpha+1$,
\begin{equation}\label{EQ JL7}
\vert F_\alpha^{(k)}(\xi)\vert \leq e^{-c|\xi|}\sum_{l=0}^{k}A_{k,l,\alpha}c^{k-l}\mathcal{L}[t^\alpha(t+2|\xi|)^{\alpha-l}](c),
\end{equation}
where $A_{k,l,\alpha}$ are constants independent of $c$.
\end{lemma}

\begin{proof}
Applying the Leibniz rule to \eqref{EQ JL2}, we see that
\begin{align*}
\vert F_\alpha^{(k)}(\xi)\vert = & e^{-c|\xi|}\sum_{l=0}^{k}A_{k,l,\alpha}c^{k-l}\dfrac{d^l}{d\xi^l}\left(\mathcal{L}[t^\alpha(t+2|\xi|)^{\alpha}](c) \right)(\xi).
\end{align*}
Differentiating under the integral sign, which is justified by the exponential decay of the integrand, and using the triangle inequality yield \eqref{EQ JL7}. 
\end{proof}

\begin{remark}
Lemmas \ref{LEM JL2} and \ref{LEM JL3} combine to show that for $0 \leq k < 2\alpha + 1$, $F^{(k)}_\alpha$ has no singularities, while for $k=2\alpha+1$, a logarithmic singularity is introduced at the origin.
\end{remark}

Our next lemma establishes pointwise estimates for the derivatives of both $\widehat{\phi_{c}}$ and $1/\widehat{\phi_{c}}$.

\begin{lemma}\label{LEM JL4}
Suppose $\alpha>0$ and $c\geq 1$.  If $0\leq k \leq 2\alpha+1$, then we have
\begin{itemize}
\item[(i)] $|\widehat{\phi_{c}}^{(k)}(\xi)| \leq c^{2\alpha+1}|\xi|^{-2\alpha-1-k}e^{-c|\xi|}  \finsum{l}{0}{k} \finsum{l'}{0}{l}A_{k,l,l',\alpha}c^{l-l'}|\xi|^{l}\mathcal{L}[t^\alpha(t+2|\xi|)^{\alpha-l'}](c)$, and
\item[(ii)] $|(1/\widehat{\phi_{c}})^{(k)}(\xi)|\leq  A_{k,\alpha} c^{k(2\alpha+1-\lfloor\alpha\rfloor)}e^{c|\xi|} |\xi|^{2\alpha+1-k} \left[ 1 + O(1)     \right]$, as $|\xi|\to 0$.
\end{itemize}
Here $A_{k,l,l',\alpha}$ and $A_{k, \alpha}$ are positive constants.
\end{lemma}
\begin{proof}
To see (i), apply the Leibniz rule, triangle inequality, and Lemma \ref{LEM JL3} to \eqref{EQ JL3}.  As a prelude to (ii), we remark that combining (i) with Lemma \ref{LEM JL2} (i) and (ii) reveals that if $0\leq k < 2\alpha+1$,  
\begin{equation}\label{EQ JL8}
|\widehat{\phi_{c}}^{(k)}(\xi)| \leq A_\alpha c^{2\alpha-\lfloor\alpha\rfloor+k}|\xi|^{-2\alpha-1-k}e^{-c|\xi|} \left[1+ O(1)   \right], \quad |\xi|\to 0.
\end{equation}
If $k=2\alpha+1$, we pick up an extra singularity from Lemma \ref{LEM JL2} (iii), in which case we split the double sum  in (i) into three parts, the first of which corresponds to $l'=l=k$, the second to $l'<l , l=k $, and the third to $l'\leq l, l<k$:
%\begin{align} 
%\nonumber |\widehat{\phi_{c}}^{(k)}(\xi)| \leq & A_{\alpha}c^{2\alpha-\lfloor\alpha\rfloor}e^{-c|\xi|} |\xi|^{-2\alpha-1}\ln(1+|\xi|^{-1}) \\
%\label{EQ JL9}+ & c^{2\alpha-\lfloor\alpha\rfloor}|\xi|^{-4\alpha-2}e^{-c|\xi|}\sum_{(l,l')\in \Lambda_k} A_{\alpha,l,l'}|\xi|^{l}\mathcal{L}[t^\alpha(t+2|\xi|)^{\alpha-l'}](c),
%\end{align}
%where $\Lambda_k = \{ (l,l')\in \mathbb{N}^2_{0} : 0\leq l \leq k , 0\leq l' \leq l , l'\neq k     \}$.

\begin{align}
\nonumber |\widehat{\phi_c}^{(k)}(\xi)|\leq & A_\alpha c^{2\alpha+1}e^{-c|\xi|}|\xi|^{-4\alpha-2}\sum_{l=0}^{k}\sum_{l'=0}^{l}A_{k,l,l'}c^{l-l'}|\xi|^{l}\mathcal{L}[t^\alpha(t+2|\xi|)^{\alpha-l'}](c)\\
\nonumber =&  A_\alpha c^{2\alpha+1}e^{-c|\xi|}|\xi|^{-4\alpha-2} \left( |\xi|^{2\alpha+1}\ln(1+|\xi|^{-1}) \right.\\
\nonumber +& \left. \sum_{l'=0}^{k-1}A_{k,l'}c^{2\alpha+1-l'}|\xi|^{2\alpha+1}\mathcal{L}[t^\alpha(t+2|\xi|)^{\alpha-l'}](c)\right. \\
\nonumber +& \left.\sum_{l=0}^{k-1}\sum_{l'=0}^{l}A_{k,l,l'}c^{l-l'}|\xi|^{l}\mathcal{L}[t^\alpha(t+2|\xi|)^{\alpha-l'}](c) \right)\\
\label{EQ JL9}\leq& A_\alpha c^{2\alpha+1}e^{-c|\xi|}|\xi|^{-2\alpha-1}\ln(1+|\xi|^{-1})\nonumber\\ +& A_\alpha c^{4\alpha+1-\lfloor\alpha\rfloor}e^{-c|\xi|}|\xi|^{-4\alpha-2}[1+O(1)], \quad |\xi|\to 0.
\end{align}
To track the largest powers of $c$, we note that there is nothing to do in the logarithmic term; in the second term we use Lemma \ref{LEM JL2} and take $l'=\lfloor\alpha\rfloor$ to obtain $c^{4\alpha+1-\lfloor\alpha\rfloor}$; for the third term we let $l=2\alpha$ and $l'=\lfloor\alpha\rfloor$ to obtain $c^{4\alpha-\lfloor\alpha\rfloor}$.

To prove (ii), we again apply the Leibniz rule to obtain
\begin{equation}\label{EQ JL10}
\left(1/\widehat{\phi_{c}}\right)^{(k)}(\xi) = \left(\widehat{\phi_{c}}(\xi) \right)^{-k-1} \sum_{\gamma\in \Gamma_k} A_\gamma \prod_{l=1}^{k}\widehat{\phi_{c}}^{(\gamma_l)}(\xi),
\end{equation}
where $\Gamma_k$ is the set of increasing non-negative integer partitions of $k$, that is,  $\Gamma_k = \{ (\gamma_1,\dots,\gamma_k)\in \mathbb{N}_{0}^{k} :  \gamma_l \leq \gamma_{l+1}, \sum \gamma_l =k    \}$.  For $0\leq k < 2\alpha+1$, we may plug \eqref{EQ JL3}, \eqref{EQ JL7}, and \eqref{EQ JL8} into \eqref{EQ JL10} to obtain:
\begin{align}
\nonumber\left| \left(1/\widehat{\phi_{c}} \right)^{(k)}(\xi)  \right| \leq & A_\alpha c^{-(2\alpha+1)(k+1)}|\xi|^{(2\alpha+1)(k+1)}\left( F_\alpha(\xi) \right)^{-k-1}\\
\times &\sum_{\gamma\in\Gamma_k}A_\gamma \prod_{l=1}^{k} c^{2\alpha+\gamma_l-\lfloor\alpha\rfloor}|\xi|^{-2\alpha-1-\gamma_l}e^{-c|\xi|}[1+O(1)], \quad|\xi|\to 0.
\end{align}

Now applying \eqref{EQ JL6} and collecting terms provides the desired estimate 
\begin{equation}\label{EQ JL11}
\left| \left(1/\widehat{\phi_{c}} \right)^{(k)}(\xi)  \right| \leq  A_{k,\alpha}c^{k(2\alpha+1-\lfloor\alpha\rfloor)}e^{c|\xi|}|\xi|^{2\alpha+1-k}[1+O(1)], \quad |\xi|\to 0.
\end{equation}

%we may use Lemma \ref{LEM JL1}, and \eqref{EQ JL8} to obtain
%\begin{align}
%\nonumber\left| \left(1/\widehat{\phi_{c}} \right)^{(k)}(\xi)  \right| 
%\leq & \dfrac {c^{-2\alpha-1-k\lfloor\alpha\rfloor}|\xi|^{2\alpha+1-k}  e^{-ck|\xi|} }{ \left( F_\alpha(\xi) \right)^{k+1} }\times\\ \nonumber\sum_{\gamma\in\Gamma_k} & A_\gamma \prod_{l=1}^{k} \left( \sum_{l'=0}^{\gamma_l} \sum_{l''=0}^{l'}A_{l,l',l'',\alpha}|\xi|^{l'}\mathcal{L}[t^\alpha(t+2|\xi|)^{\alpha-l''}](c)      \right) \\
%\leq & A_{k, \alpha} c^{k(2\alpha-\lfloor \alpha\rfloor+1)}e^{c|\xi|} |\xi|^{2\alpha+1-k} \left[ 1 + O(1)     \right], \quad |\xi|\to 0.\label{EQ JL11}
%\end{align}

If $k=2\alpha+1$, then the logarithmic singularity in \eqref{EQ JL9} appears only when $\gamma = (0,\dots,0,k)$.  By using \eqref{EQ JL6}, we see that the corresponding term satisfies

\begin{align}
\nonumber \left|\dfrac{\widehat{\phi_{c}}^{(k)}(\xi)}{(\widehat{\phi_{c}}(\xi))^2} \right|\leq & A_\alpha c^{2\alpha+1}e^{c|\xi|}|\xi|^{2\alpha+1}\ln(1+|\xi|^{-1})+A_\alpha c^{4\alpha+1-\lfloor\alpha\rfloor}e^{c|\xi|}[1+O(1)]\\
\label{EQ JL12}=& A_{\alpha} c^{2\alpha+1}e^{c|\xi|}o(1)+A_{\alpha} c^{4\alpha+1-\lfloor\alpha\rfloor}e^{c|\xi|}[1+O(1)], \quad |\xi|\to 0.
\end{align}
Thus the term containing the logarithmic singularity is bounded by the estimate in \eqref{EQ JL11}, and the proof is complete.

%\begin{equation}\label{EQ JL12}
%\left|\dfrac{\widehat{\phi_{c}}^{(k)}(\xi)}{(\widehat{\phi_{c}}(\xi))^2} \right| \leq A_{\alpha} c^{2\alpha+1} e^{c|\xi|}|\xi|^{2\alpha+1}\ln(1+|\xi|^{-1}) = o(1),\quad |\xi|\to 0.
%\end{equation}
%The final equality follows from the fact that $\alpha>0$.

\end{proof}

\begin{remark}
The calculations above show that for $0\leq k \leq 2\alpha+1$, $\left( 1/\widehat{\phi_{c}} \right)^{(k)}\in L_\infty[-\pi,\pi]$.
\end{remark}

We are now in position to prove Proposition \ref{PROP l1}.

\begin{proof}[Proof of Proposition \ref{PROP l1}]
Recall that for $j\neq 0$, $a_j(\xi)=\widehat{\phi_{c}}(\xi+2\pi j)/\widehat{\phi_{c}}(\xi)$.  From the Leibniz rule, we have
\begin{equation}\label{EQ JL13}
a_j^{(k)}(\xi)= \sum_{l=0}^{k}A_{k,l} \left( 1/\widehat{\phi_{c}}  \right)^{(l)}(\xi) \widehat{\phi_c}^{(k-l)}(\xi+2\pi j).
\end{equation}
Thus we may use \eqref{EQ JL8} and \eqref{EQ JL11} to obtain
\[
| a_j^{(k)}(\xi)  | \leq g(\xi)\sum_{l=0}^{k} A_{k,l} c^{(l+1)(2\alpha-\lfloor\alpha\rfloor)+k}e^{c|\xi|-c|\xi-2\pi j|} |\xi|^{2\alpha+1-l}|\xi+2\pi j|^{-(2\alpha+1+k-l)},
\]
where $g\in L_\infty[-\pi,\pi]$.  Since $|\xi|\leq \pi(1-\varepsilon)$, $|\xi+ 2\pi j|-|\xi|\geq 2\pi(|j|-1+\varepsilon)$, and $|\xi+2\pi j|\geq \pi(1+\varepsilon)$, we have
\[
| a_j^{(k)}(\xi)  |\leq c^{(k+1)(2\alpha-\lfloor\alpha\rfloor)+k}e^{-2\pi c \varepsilon}e^{-2\pi c (|j|-1)}\sum_{l=0}^{k}A_{k,l,\alpha}\dfrac{(\pi(1-\varepsilon))^{2\alpha+1-l}}{(\pi(1+\varepsilon))^{2\alpha+1+k-l}}.
\]
This is the desired estimate for $1\leq k <2\alpha+1$.  If $k=2\alpha+1$, we must use \eqref{EQ JL9} and \eqref{EQ JL12} to handle the logarithmic term.  This changes $A_{\alpha}(\varepsilon)$; however, it is still true that $A_{\alpha}(\varepsilon)=O(1)$ as $\varepsilon\to 0$.
\end{proof}

An immediate consequence of Proposition \ref{PROP l1} is that the function $s_c$ defined in \eqref{EQ l2} converges for $|\xi|\leq \pi(1-\varepsilon)$ and we can differentiate the series term by term.  In fact, by applying the estimates from \eqref{EQ l3}, we have 
\begin{equation}\label{EQ JL14}
|s_c^{(k)}(\xi)  |\leq A_{k,\alpha}(\varepsilon) c^{(k+1)(2\alpha-\lfloor\alpha\rfloor)+k}e^{-2\pi c \varepsilon},   
\end{equation}
for $0\leq k\leq 2\alpha+1$ and $|\xi|\leq\pi(1-\varepsilon)$, where $A_{k,\alpha}(\varepsilon)=O(1)$ as $\varepsilon\to 0$.

Recalling that $\widehat{L_{c}}(\xi)=(1+s_c(\xi))^{-1}$, we may prove pointwise bounds for $\widehat{L_{c}}^{(k)}$ by applying \eqref{EQ JL14} in a similar manner to \eqref{EQ JL10}.  The next three lemmas prove Proposition \ref{PROP l2}.

\begin{lemma}\label{LEM JL5}
Let $\varepsilon\in[0,1)$ and suppose that $|\xi|\leq \pi(1-\varepsilon)$.  If $0\leq k\leq 2\alpha+1$, then 
\begin{equation}\label{EQ JL15}
|\widehat{L_{c}}^{(k)}(\xi)|\leq A_{k,\alpha}(\varepsilon)c^{2k(2\alpha-\lfloor\alpha\rfloor)+k}e^{-2\pi c \varepsilon},
\end{equation} 
where $A_{k,\alpha}(\varepsilon)=O(1)$ as $\varepsilon\to 0$.
\end{lemma}

Analogous to the argument given in \cite{RiemSiva}, we find that if $|j|\geq 2$ and $\xi\in [(-2j-1)\pi,(-2j+1)\pi]$, then
\[
\widehat{L_{c}}(\xi)=a_{-j}(r)\widehat{L_{c}}(r),
\]
where $r=2\pi j +\xi$.  This means that $r\in [-\pi,\pi]$, so we may use the Leibniz rule together with Proposition \ref{PROP l1} and Lemma \ref{LEM JL5} (letting $\varepsilon =0$) to obtain our next result.

\begin{lemma}\label{LEM JL6}
Let $0\leq k \leq 2\alpha+1$.  If $|j|\geq 2$ and $\xi\in [(-2j-1)\pi,(-2j+1)\pi]$, then
\begin{equation}\label{EQ JL16}
\vert  \widehat{L_{c}}^{(k)}(\xi)  \vert \leq A_{k,\alpha} c^{(2k+1)(2\alpha-\lfloor\alpha\rfloor)+k}e^{-2\pi c (|j|-1)},
\end{equation} 
where $A_{k,\alpha}$ is independent of both $c$ and $j$.
\end{lemma}
Our final estimate is for the region $|\xi|\in [(1+\varepsilon)\pi,3\pi]$.  For these intervals, we adapt the argument given for Theorem 2.4 in \cite{RiemSiva}.

\begin{lemma}\label{LEM JL7}
Let $\varepsilon\in [0,1)$ and $0\leq k\leq 2\alpha+1$.  If $|\xi|\in [(1+\varepsilon)\pi,3\pi]$, then
\begin{equation}\label{EQ JL17}
\vert \widehat{L_{c}}^{(k)} (\xi) \vert \leq A_{k,\alpha}(\varepsilon) c^{(2k+1)(2\alpha-\lfloor\alpha\rfloor)+k}e^{-\pi c \varepsilon},
\end{equation}
where $A_{k,\alpha}(\varepsilon)=O(1)$ as $\varepsilon\to 0$.
\end{lemma}

\begin{proof}[Proof of Theorem \ref{THM l1}]
Applying the pointwise estimates from Proposition \ref{PROP l2} establishes the $L_1$ bound. 
\end{proof}

\begin{proof}[Proof of Theorem \ref{THM l2}]
Using \eqref{EQ JL4} and replacing $F_\alpha$ with $F_{|\alpha|-1}$ one obtains the stated bounds by using reasoning similar to that used to establish Theorem \ref{THM l1}.
\end{proof}

\begin{proof}[Proof of Theorem \ref{THM KPoissonFTdecay}]
In the special case of $\alpha=-1$, we get the Poisson kernel, whose Fourier transform is given by 
\[
\widehat{\phi_c}(\xi)=A F_0(\xi)=(A/c)e^{-c|\xi|}.
\]
Using this much simpler formula allow us to simplify our earlier work and get stronger results.  In fact, we see that the analogue of Proposition 4.1 is given by
\[
|a_j^{(k)}(\xi)|\leq A c^{k}e^{c|\xi|-c|\xi+2\pi j|},
\]
where there is no longer a restriction on $k\in\mathbb{N}_0$.  The proof of Theorem \ref{THM KPoissonFTdecay} now follows the same line of reasoning as that used to prove Theorem \ref{THM l1}.
\end{proof}

\begin{proof}[Proof of Theorem \ref{THM l3}]
We begin by noting that $\widehat{L_c}$ is even, so we need only consider the integral
\[
\int_{0}^{\infty}|\widehat{L_c}' (\xi)|d\xi = \int_{0}^{\pi/2}|\widehat{L_c}' (\xi)|d\xi+\int_{\pi/2}^{3\pi/2}|\widehat{L_c}' (\xi)|d\xi+\int_{3\pi/2}^{\infty}|\widehat{L_c}' (\xi)|d\xi =: I+II+III.
\]
For $I$, we use Lemma \ref{LEM JL5} with $\varepsilon = 1/2$ and find that $I \leq A_\alpha c^{2(2\alpha-\lfloor\alpha\rfloor)+1}e^{-\pi c} \leq A'_\alpha$.
The quantity $III$ may be estimated similarly using Lemmas \ref{LEM JL6} and \ref{LEM JL7} (and again letting $\varepsilon =1/2$).  We have
\begin{align*}
III \leq& A_\alpha c^{3(2\alpha-\lfloor\alpha\rfloor)+1} \left(  e^{-\pi c /2} + \sum_{j=1}^{\infty}e^{-2\pi c j}       \right)  \\
\leq &  A_\alpha \left(    c^{3(2\alpha-\lfloor\alpha\rfloor)+1}e^{-\pi c /2}    +  \dfrac{ c^{3(2\alpha-\lfloor\alpha\rfloor)+1}e^{-2\pi c }}{1-e^{-2\pi c}}         \right)   \leq A'_\alpha           
\end{align*}
The last inequality follows from the fact that both terms in parentheses have a global maximum which depends only on $\alpha$.
Finally, we show that $\widehat{L_c}$ is monotone on $[\pi/2,3\pi/2]$, which implies that $II=|\widehat{L_c}(\pi/2) - \widehat{L_c}(3\pi/2)| \leq 1$.  To that end, we use the quotient rule to write
\[
\widehat{L_{c}}'(\xi) = \dfrac{\dsum_{j\neq 0} \left\{ \widehat{\phi_c}'(\xi)\widehat{\phi_c}(\xi+2\pi j) -  \widehat{\phi_c}(\xi)\widehat{\phi_c}'(\xi+2\pi j)      \right\}      }{\left(  \dsum_{j\in\mathbb{Z}}  \widehat{\phi_c}(\xi+2\pi j)          \right)^2}.
\]
%Without loss of generality, we assume $A_\alpha>0$ and continue.
We recall that $\widehat{\phi_c}(\xi) = A_{\alpha}c^{\alpha+1/2}|\xi|^{-\alpha-1/2}K_{\alpha+1/2}(c|\xi|),$ and by the formula found in \cite[p. 361]{AandS}, $\widehat{\phi_c}'(\xi)=-A_{\alpha}c^{\alpha+3/2}\sgn(\xi)|\xi|^{-\alpha-1/2} K_{\alpha+3/2}(c|\xi|)$.  These allow us to rewrite the numerator,
\begin{multline}
\sum_{j\neq 0} \left\{ \widehat{\phi_c}'(\xi)\widehat{\phi_c}(\xi+2\pi j) -  \widehat{\phi_c}(\xi)\widehat{\phi_c}'(\xi+2\pi j) \right\} 
 =  A_{\alpha}^2c^{2\alpha+2}|\xi|^{-\alpha  -1/2} \\\times \left\{g_1(\xi) K_{\alpha+1/2}(c|\xi|)  - g_2(\xi)K_{\alpha+3/2}(c|\xi|)       \right\},
\end{multline}
where
\begin{align*}
g_1(\xi)& = \sum_{j=1}^{\infty} \left(  (2\pi j + \xi)^{-\alpha-1/2}K_{\alpha+3/2}(c(2\pi j +\xi)) \right.\\ &-\left. (2\pi j - \xi)^{-\alpha-1/2}K_{\alpha+3/2}(c(2\pi j -\xi))    \right),   \\
g_2(\xi)& = \sum_{j=1}^{\infty}\left(  (2\pi j + \xi)^{-\alpha-1/2}K_{\alpha+1/2}(c(2\pi j +\xi)) \right.\\ &+\left. (2\pi j - \xi)^{-\alpha-1/2}K_{\alpha+1/2}(c(2\pi j -\xi))    \right). 
\end{align*}
From these expressions, we see that $g_2(\xi)>0$ for $\xi\in [\pi/2,3\pi/2]$ and $g_1(\xi)<0$ on $[\pi/2,3\pi/2]$.  The latter is true since $f(x)=x^{-\alpha-1/2}K_{\alpha+3/2}(x)$ is decreasing.  This shows that $\widehat{L_c}$ is decreasing on $[\pi/2,3\pi/2]$ as desired. 
%If $A_\alpha<0$, we see that $\widehat{L_c}$ is increasing on $[\pi/2,3\pi/2]$. 
\end{proof}

\begin{proof}[Proof of Lemma  \ref{LEM h3}]
Recall that $\widehat{L_c}$ is non-negative as commented in Section \ref{SECbaxteranalogue}, and Theorems \ref{THM l1} and \ref{THM l2} show that $L_c$ is continuous and integrable for the given range of $\alpha$.  Consequently, $L_c$ is positive definite, and thus for every $x\in\R$,
$$|L_c(x)|\leq|L_c(0)|=1,$$
the final equality coming from \eqref{EQHfundamental}.
\end{proof}

We end the section with the proof of Theorem \ref{THM h3}.

\begin{proof}[Proof of Theorem \ref{THM h3}]
We first consider the case $k=0$.  As in Section 3, write $\widehat{L_c}(\xi) = (1+\sum_{j\neq0}a_j(\xi))^{-1}$.  Proposition \ref{PROP l1} with $\eps=0$ implies that the series converges uniformly for $|\xi|\leq\pi$.  Consequently, $\widehat{L_c}(0) = \underset{\xi\to0}\lim\;\widehat{L_c}(\xi) = 1$.  To see that this limit is 1 requires estimating $\widehat{L_c}(\xi)$ in a slightly different manner than we have so far.  Combining \eqref{EQ JL3} and \eqref{EQ JL5}, we estimate
$$\zsum{j}|a_j(\xi)|\leq\zsumzero{j}\left|\dfrac{\xi}{\xi+2\pi j}\right|^{2\alpha+1}e^{-c(|\xi+2\pi j|-|\xi|)}\left|\dfrac{\mathcal{L}[t^\alpha(t+2|\xi+2\pi j|)^\alpha](c)}{\mathcal{L}[t^\alpha(t+2|\xi|)^\alpha](c)}\right|,$$
which by Lemma \ref{LEM JL1} and Proposition \ref{LEM JL2}(i) is at most
$$|\xi|^{2\alpha+1}\zsumzero{j}\dfrac{1}{|\xi+2\pi j|^{2\alpha+1}}e^{-c(|\xi+2\pi j|-|\xi|)}\left(\dfrac{4\Gamma(\alpha+1)}{\Gamma(2\alpha+1)}c^\alpha|\xi+2\pi j|^\alpha+2^\alpha\right).$$
The series is summable, and uniformly bounded for $|\xi|\leq\pi$, and it follows that $\underset{\xi\to0}\lim\;\widehat{L_c}(\xi) = 1$.

Now for $|k|\geq1$, let $r = 2\pi k+\xi$ with $|\xi|\leq\pi$.  Then we may write
\begin{align*}\widehat{L_c}(r) & = 
\dfrac{\widehat{\phi_c}(\xi+2\pi k)}{\widehat{\phi_c}(\xi)\left[1+\dsum_{j\neq-k}\frac{\widehat{\phi_c}(\xi+2\pi(k+j))}{\widehat{\phi_c}(\xi)}\right]}\\  &= \dfrac{|\xi|^{2\alpha+1}\widehat{\phi_c}(\xi+2\pi k)}{A_\alpha c^{2\alpha+1}F_\alpha(\xi)\left[1+\dfrac{c^{-2\alpha-1}|\xi|^{2\alpha+1}}{F_\alpha(\xi)}\zsumzero{j}\widehat{\phi_c}(\xi+2\pi j)\right]},
\end{align*}
the second equality following from \eqref{EQ JL3}.  By definition, $\underset{\xi\to0}\lim\,F_\alpha(\xi) = F_\alpha(0)\neq0$, and a similar argument to the case $k=0$ above shows that we may allow $\xi$ to tend to 0, and conclude that $\widehat{L_c}(2\pi k)=0$.
\end{proof}


\begin{thebibliography}{1}

\bibitem{AandS}
M. Abramowitz and I. A. Stegun (Eds.), {\em Handbook of mathematical functions: with formulas, graphs, and mathematical tables.} No. 55, Courier Dover Publications, 1972.

\bibitem{BSW} K. Ball, N. Sivakumar, and J. D. Ward, On the sensitivity of radial basis interpolation to minimal data separation distance, {\em Constr. Approx.} \textbf{8}(4) (1992), 401-426.

\bibitem{Baxter}
B. J. C. Baxter, The asymptotic cardinal function of the multiquadratic
$\phi(r)=(r^2+ c^2)^\frac{1}{2}$ as $c\to\infty$, {\em Comput. Math. Appl.}, \textbf{24}(12), (1992), 1-6.

\bibitem{BS}
B. J. C. Baxter and N. Sivakumar, On shifted cardinal interpolation by Gaussians and multiquadrics, {\em J. Approx. Theory } {\bf 87} (1996), no. 1, 36-59.

\bibitem{BLM} R. K. Beatson, J. Levesley, and C. T. Mouat, Better bases for radial basis function interpolation problems,  {\em J. Comput. Appl. Math.} \textbf{236}(4) (2011), 434-446.

\bibitem{BennettSharpley}
C. Bennett and R. Sharpley, {\em Interpolation of Operators}, New York, 1988. 

\bibitem{Buhmann}
M. Buhmann, Multivariate cardinal interpolation with radial-basis functions,
{\em Constr. Approx.} \textbf{6.3} (1990), 225-255.

\bibitem{Buhmann_compact_support}
M. Buhmann, Radial functions on compact support, {\em Proc. Edinburgh Math. Soc. (2)}, {\bf 41}, no. 1 (1998), 33-46. 

\bibitem{Buhmannbook}
M. Buhmann, {\em Radial basis functions: theory and implementations}, Vol. 12, Cambridge University Press, 2003.

\bibitem{buhmannmicchelli}
M. Buhmann and C. A. Micchelli, Multiquadric interpolation improved, {\em Comput. Math. Appl.}, \textbf{24}(12), (1992), 21-25.

\bibitem{DF}
T. Driscoll and B.Fornberg, Interpolation in the limit of increasingly flat radial basis functions, {\em Comput. Math. Appl.}, {\bf 43} (2002), 413-422.

\bibitem{FM} G. E. Fasshauer and M. J. McCourt, Stable evaluation of Gaussian radial basis function interpolants, {\em SIAM J. Sci. Comput.} \textbf{34}(2) (2012), A737-A762.

\bibitem{FornFlyerBook} B. Fornberg and N. Flyer, {\em A Primer on Radial Basis Functions with Applications to the Geosciences}, Vol. 87. SIAM, 2015.

\bibitem{FFL} B. Fornberg, E. Larsson, and N. Flyer, Stable computations with Gaussian radial basis functions, {\em SIAM J. Sci. Comput.} \textbf{33}(2) (2011), 869-892.

\bibitem{FL} E. Larsson and B. Fornberg, Theoretical and computational aspects of multivariate interpolation with increasingly flat radial basis functions, {\em Comput.  Math. Appl.} \textbf{49}(1) (2005), 103-130.

\bibitem{FW} B. Fornberg and G. B. Wright, Stable computation of multiquadric interpolants for all values of the shape parameter, {\em Comput. Math. Appl.} \textbf{48}(5) (2004), 853-867.

\bibitem{FHNWW} E. Fuselier, T. Hangelbroek, F. J. Narcowich, J. D. Ward, and G. B. Wright, Localized bases for kernel spaces on the unit sphere, {\em SIAM J. Numer. Anal.} \textbf{51}(5) (2013), 2538-2562.

\bibitem{HMNW} T. Hangelbroek, W. Madych, F. Narcowich and J. Ward, Cardinal
interpolation with Gaussian kernels, {\em J. Fourier Anal. Appl.} \textbf{18}
(2012), 67-86.

\bibitem{HNRW} T. Hangelbroek, F.J. Narcowich, C. Rieger, and J. D. Ward, An inverse theorem on bounded domains for meshless methods using localized bases. arXiv preprint arXiv:1406.1435 (2014).

\bibitem{HNSW1} T. Hangelbroek, F. J. Narcowich, and J. D. Ward, Kernel approximation on manifolds I: Bounding the Lebesgue constant, {\em SIAM J. Math. Anal.} \textbf{42}(4) (2010), 1732-1760.

\bibitem{HNSW2} T. Hangelbroek, F. J. Narcowich, X. Sun, and J. D. Ward, Kernel approximation on manifolds II: the $L_\infty$ norm of the $L_2$ projector, {\em SIAM J. Math. Anal.} \textbf{43}(2) (2011), 662-684.

\bibitem{ka}
Y. Katznelson, {\em An Introduction to Harmonic Analysis, Third Edition}, Dover, 2004.

\bibitem{Ledford}
J. Ledford, On the convergence of regular families of cardinal interpolators,
{\em Adv. Comput. Math.} {\bf 41} (2015), no. 2, 357-371.

\bibitem{Ledford_ellp}
J. Ledford, Convergence properties of spline-like cardinal interpolation operators acting on $\ell^p$ data, (submitted) arXiv:1312.4062, 2013. 

\bibitem{MadychMisc} W. R. Madych, Miscellaneous error bounds for multiquadric and related interpolators, {\em Comput. Math. Appl.} \textbf{24}(12) (1992), 121-138.

\bibitem{MRR} 
M. J. Marsden, F. B. Richards, and S. D. Riemenschneider, Cardinal spline interpolation operators on $\ell^p$ data, {\em Indiana Univ. Math. J.} \textbf{24} (1975), 677-689; Erratum, {\em ibid.}, \textbf{25} (1976), 919.

\bibitem{NRW} F.J. Narcowich, S. T. Rowe, and J. D. Ward, A novel Galerkin method for solving PDEs on the sphere using highly localized kernel bases, arXiv preprint, arXiv:1404.5263 (2014).

\bibitem{RS2} 
S. D. Riemenschneider and N. Sivakumar, Cardinal interpolation by Gaussian functions: a survey, {\em J. Anal.} \textbf{8} (2000), 157-178.

\bibitem{RiemSiva} 
S. D. Riemenschneider and N. Sivakumar, On the Cardinal-interpolation Operator Associated with the One-dimensional
Multiquadric, {\em East J. Approx.} \textbf{7}, no. 4 (1999), 485-514.

\bibitem{RS3} 
S. D. Riemenschneider and N. Sivakumar, On cardinal interpolation by Gaussian radial-basis functions: properties of fundamental functions and estimates for Lebesgue constants, {\em J. Anal. Math.} {\bf 79} (1999), 33-61.

\bibitem{Schoenberg} I. J. Schoenberg, {\em Cardinal Spline Interpolation},  Vol. 12. Society for Industrial and Applied Mathematics, Philadelphia, 1973.

\bibitem{Watson} 
G. N. Watson, {\em A treatise on the theory of Bessel functions, Second Edition}, Cambridge University Press, 1944.

\bibitem{Wendland_ACOM}
H. Wendland, Piecewise polynomial, positive definite and compactly supported radial functions of minimal degree, {\em Adv. Comput. Math.} {\bf 4} (1995), 389-396.

\bibitem{Wendland}
H. Wendland, {\em Scattered data approximation}, Vol. 17, Cambridge
University Press, 2005.

\end{thebibliography}
\end{document}